\documentclass[10pt,a4paper,reqno]{amsart}

\usepackage{microtype}
\usepackage[dvipsnames]{xcolor}
\usepackage{amssymb}
\usepackage{amsmath}
\usepackage{amsthm}
\usepackage{amsfonts}
\usepackage{thmtools}
\usepackage[colorlinks = true, linkbordercolor = {white}, urlcolor={purple}, linkcolor={purple}, citecolor = {purple}]{hyperref}

\usepackage[foot]{amsaddr}
\usepackage{pifont}
\usepackage{graphicx}
\usepackage{mathtools}

\usepackage{framed}
\usepackage{longtable}
\usepackage{pgfplots}
\pgfplotsset{width=10cm,compat=1.18}

\definecolor{webgreen}{rgb}{0,.5,0}
\definecolor{webbrown}{rgb}{.6,0,0}

\usepackage{fullpage}
\usepackage{float}
\usepackage{subcaption}
\usepackage{graphics}
\usepackage{latexsym}
\usepackage{tikz}
\usetikzlibrary{automata, positioning, arrows, arrows.meta, shapes.misc, decorations.markings}

\newcommand{\seqnum}[1]{\href{https://oeis.org/#1}{\rm \underline{#1}}}
\newcommand{\Z}{{\mathbb Z}}

\newcommand{\R}{{\mathbb R}}
\newcommand{\N}{{\mathbb N}}

\newcommand{\bff}{{\mathbf f}}
\newcommand{\bfx}{{\mathbf x}}

\newcommand{\bft}{{\mathbf t}}
\newcommand{\bfr}{{\mathbf r}}
\newcommand{\mc}{\mathcal}
\newcommand{\A}{\mc A}
\newcommand{\T}{\mc T}
\newcommand{\walnut}{{\ttfamily Walnut}}
\newcommand{\true}{{\ttfamily TRUE}}
\newcommand{\false}{{\ttfamily FALSE}}

\begin{document}

\theoremstyle{definition}

\newtheorem{theorem}{Theorem}[section]
\newtheorem{definition}[theorem]{Definition}
\newtheorem{cor}[theorem]{Corollary}
\newtheorem{prop}[theorem]{Proposition}
\newtheorem{lem}[theorem]{Lemma}
\newtheorem{example}[theorem]{Example}
\newtheorem{que}[theorem]{Question}
\newtheorem{conj}[theorem]{Conjecture}
\theoremstyle{remark}
\newtheorem{rem}[theorem]{Remark}

\title{Monochromatic arithmetic progressions in the {F}ibonacci, {T}hue--{M}orse, and {R}udin--{S}hapiro words}
 
\author[G.\,Joshi]{G.\,Joshi}
\author[D.\,Rust]{D.\,Rust}
\email{gandhar.joshi@open.ac.uk, dan.rust@open.ac.uk }

\address[G.\,Joshi]{School of Mathematics and Statistics, The Open University, Milton Keynes, MK7 6AA, UK}
\address[D.\,Rust]{School of Mathematics and Statistics, The Open University, Milton Keynes, MK7 6AA, UK}



\begin{abstract}
We investigate the lengths and starting positions of the longest monochromatic arithmetic progressions for a fixed difference in the Fibonacci word. We provide a complete classification for their lengths in terms of a simple formula. Our strongest results are proved using methods from dynamical systems, especially the dynamics of circle rotations. We also employ computer-based methods in the form of the automatic theorem-proving software \walnut{}. This allows us to extend recent results concerning similar questions for the Thue--Morse word and the Rudin--Shapiro word. This also allows us to obtain some results for the Fibonacci word that do not seem to be amenable to dynamical methods.
\end{abstract}

\maketitle

\section{Introduction}
Consider an infinite sequence of coloured positions taken from a finite colour palette. A {\em monochromatic arithmetic progression} with constant difference is the finite or infinite repetition of a single colour in that sequence with a fixed finite difference between those positions. Equivalently, we may consider infinite words over a finite alphabet $\mathcal{A}$ and occurrences of a letter from $\mathcal{A}$ that appear in arithmetic progression within the word. Throughout this work, we write `MAP' to mean `monochromatic arithmetic progression'. Van der Waerden's celebrated theorem \cite{vdw1927} says that for any finite alphabet $\mathcal{A}$ and any natural number $k$, there is a natural number $N$ such that for any finite word over $\mathcal{A}$ of length $N$, there must exist a MAP of length $k$.

For a general infinite word, Van der Waerden's theorem only guarantees the existence of arbitrarily long MAPs, not their lengths or positions in the word. Likewise, it is not true that the existence of arbitrarily long MAPs implies the existence of infinite MAPs. These properties are therefore worth exploring for particular infinite words or families of words. In this work, we are not so concerned with the existence of infinite MAPs. Morgenbesser, Shallit, and Stoll \cite{mss2011} showed that the Thue--Morse word only admits finite MAPs as a consequence of a result of Gelfond \cite{gelfond68}. Works by Durand and Goyheneche \cite{dg2019}, and by Nagai, Akiyama and Lee \cite{nagai2022} (for tilings) address the question for general uniformly recurrent sequences (resp.\ primitive substitution tilings) from a dynamical perspective, giving conditions for the non-existence of infinite MAPs in terms of spectral properties of the subshift (resp.\ hull) associated with a sequence (resp.\ tiling).

Instead, we turn our attention to understanding how long MAPs can be for a fixed difference and the starting positions of their first occurrence. There has been recent progress in this direction for certain families of words. In particular, most of the literature has so far focused on words coming from substitutions of constant length, also known as $k$-automatic sequences. For an introduction to fixed points of substitutions on finite alphabets we refer the reader to the book by Queff\'{e}lec \cite{queffelec2010}, and for a primer on finite-state automata and automatic sequences, see the book by Allouche and Shallit \cite{AS2003}. Parshina \cite{parshina15} addressed the question of determining the lengths of longest MAPs for generalised binary Thue--Morse words. Aedo, Grimm, Nagai, and Staynova \cite{aedo2022}, and the same authors with Ma\~nibo \cite{aedo2023} performed similar studies for `Thue--Morse like' words and substitutions with certain symmetries, including the Rudin--Shapiro word. Unfortunately, they were still unable to resolve the question of classifying the lengths of longest MAPs completely for any of these words, only providing bounds for the most part and, in the best cases, exact values for very specific families of fixed differences.

The question of determining the position of the first occurrence of a MAP of maximal length for a fixed difference is less well studied. However some results exist in the literature; most notably Parshina \cite{parshina17} tackled this question for generalised binary Thue--Morse words,
and Sobolewski \cite{sobo23}, using a computational approach, was able to determine first occurrences for some fixed differences for the Rudin--Shapiro word.

Conspicuously missing from all of these studies is an attempt to understand these questions for words coming from substitutions that are not constant length\footnote{beyond the existence or non-existence of infinite MAPs, which is addressed by the results of Durand and Goyheneche \cite{dg2019}.}. This includes a lack of any results concerning perhaps the most well-known substitution, the Fibonacci substitution.
This is likely due to the fact that the methods brought to bear in tackling these problems rely heavily on the $k$-automatic structure of the words associated with substitutions of constant length, which explains why the strongest results in this setting are for when the fixed difference is close to a power of the length of the substitution (e.g., powers of $2$ for the Thue--Morse substitution).

Our primary focus in this article is the study of the Fibonacci word. We make use of two novel approaches that prove to be exceptionally powerful: the first is that we make use of the fact that the Fibonacci word can be described in terms of the itinerary of an orbit of an irrational rotation on a circle, allowing us to use results from Diophantine approximation and dynamical systems; the second is that we take advantage of the fact that, although the Fibonacci word is not $k$-automatic, it is automatic with respect to the \emph{Zeckendorf} numeration system, which allows us to use new computational technologies in the form of automatic theorem-proving software.

{\walnut} is a software package written by Hamoon Mousavi \cite{mous_walnut} designed to automatically prove propositions concerning automatic sequences that can be stated within the first-order theories of extended Presburger arithmetic or B\"uchi arithmetic. This includes sequences that are automatic with respect to non-uniform numeration systems such as Zeckendorf numeration.

 Together, these two methods allow us to completely resolve the question of which lengths of MAPs exist for a fixed difference, culminating in two simple formulae for the maximal length (Theorems \ref{THM:adformula} and \ref{THM:adformula2}). We also make significant progress in understanding the positions of first occurrences of  MAPs of maximal length. In some instances, we provide proofs of statements using both methods, but it is worth noting that both methods have limitations and so our work benefits from a blending of the two.

The paper is organised as follows. Section \ref{SEC:prelims} provides the necessary preliminaries for substitutions, monochromatic arithmetic progressions, finite state automata, and \walnut{}. Section \ref{SEC:results-walnut} is devoted to results obtained from the use of \walnut{}. These include extensions of previous results concerning MAPs in the Thue--Morse and Rudin--Shapiro words, as well as our first results that address the Fibonacci word.
Section \ref{SEC:results-dynamics} then details our strongest results on both longest MAPs and positions of their first occurrence for the Fibonacci word. In this final section, we reinterpret the Fibonacci word as a rotation sequence on the interval $[0,1]$ and MAPs as segments of orbits of a power of the rotation. This in turn allows us to take advantage of results from the theory of dynamical systems and Diophantine approximation.

\section{Preliminaries}\label{SEC:prelims}
\subsection{Substitutions}\label{SEC:subs}
Let $\N$, $\N_0$ and $\Z$ denote the set of natural numbers, the set of non-negative integers and the set of integers, respectively. Given a finite alphabet $\mathcal A$, let $\A^n$ denote the set of words of length $n$ over $\A$ with the empty word $\varepsilon \in \A^0$, and let $\mathcal{A}^{\N_0}$ denote the set of $\N_0$-indexed infinite sequences over $\mathcal{A}$. Let $\mathcal A^{*} = \bigsqcup_{n=0}^\infty \A^n$ denote the free monoid of finite words in $\mathcal{A}$ under concatenation, and let $\mathcal{A}^+ =\bigsqcup_{n=1}^\infty \A^n$ denote the set of non-empty words in $\mathcal A$. A function $\theta\colon\mathcal{A}\to \mathcal{A}^+$ is called a {\em substitution} and naturally extends to a function $\theta\colon\A^+\to\A^+$ by concatenation of words. In this way, a substitution can be iterated by defining $\theta^1=\theta$ and $\theta^{n+1} = \theta \circ \theta^{n}$. A substitution is called {\em primitive} if there exists a power $p\geq 1$ such that for all letters $\mathtt{a,b} \in \A$, the word $\theta^p(\mathtt{b})$ contains the letter $\mathtt{a}$. Note that letters in an alphabet are written in {\ttfamily Teletype} font, and may include digits such as $\mathtt{0}$ and $\mathtt{1}$.
\begin{example}\label{fibsub}
Let $\phi$ denote the {\em Fibonacci substitution} on the alphabet $\A=\{\mathtt{0,1}\}$, defined by $\phi\colon\mathtt{0}\mapsto \mathtt{01},\, \mathtt{1}\mapsto \mathtt{0}$. This substitution is primitive, as both $\phi^2(\mathtt{0})$ and $\phi^2(\mathtt{1})$ contain both letters $\mathtt{0}$ and $\mathtt{1}$.
\end{example}

In this work, we focus on substitution fixed (resp.~substitution periodic) points. That is, sequences $\bfx$ such that $\theta(\bfx)=\bfx$ (resp.~$\theta^p(\bfx)=\bfx$ for some $p \geq 1$). For example, the {\em Fibonacci word} $\bff=\left( f_i\right)_{i\geq0}\in \{\mathtt{0,1}\}^{\N_0}$ is the unique fixed point of the Fibonacci substitution $\phi$ and is generated by taking the limit of $\phi^i(\mathtt{0})$ for increasing powers $i$, giving
\[\mathtt{0}\mapsto \mathtt{01} \mapsto \mathtt{010} \mapsto \mathtt{01001} \mapsto \cdots \mapsto \bff = \mathtt{010010100100101}\cdots.\]
Here, note that $f_i$ denotes the letter at position $i$ in the sequence $\bff$. For example, $f_4=\mathtt{1}$. 
\subsection{Monochromatic arithmetic progressions}\label{SEC:maps}
\begin{definition}
When a letter $\mathtt{a}$ in a sequence $\bfx$ repeats $\ell$ times with a difference $d$, we call it a {\em monochromatic arithmetic progression} of difference $d$ and length $\ell$. That is, for some $i \geq 0$,
\[
x_{i} = x_{i+d} = x_{i+2d} = \cdots = x_{i+(\ell-1)d}.
\]
\end{definition}
We write `MAP' to mean `monochromatic arithmetic progression'.
\begin{definition}
Let $\bfx\in \A^{\N_0}$ and $d\in\N$. We let
\[
A_\bfx(d) = \max\{\ell \mid \text{there is a MAP of difference $d$ and length $\ell$ in $\bfx$}\}
\]
denote the {\em longest length of a MAP} of difference $d$ that can be found in $\bfx$, where by convention $A_\bfx(d) = \infty$ if no finite maximum exists.
\end{definition}
\begin{rem}
    Parshina \cite{parshina15} first used the notation $A_\bfx(c,d)$ to mean the longest length of the MAP of difference $d$, where $c$ denotes the starting position of the MAP. We adopt this notation, but suppress $c$ as we consider here MAPs beginning at any position in $\bfx$.
\end{rem}
\begin{definition}
We let $i_\bfx(d)$ denote the {\em first position} at which a MAP of maximal length and of difference $d$ starts in the sequence $\bfx$. That is, \[i_\bfx(d) = \min\{j \mid x_j=x_{j+d}=\cdots = x_{j+(A(d)-1)d}\}.\]
\end{definition}
We avoid subscripts and write $A(d)$ and $i(d)$ when the context is clear.
\begin{example}\label{ex1}
In the Fibonacci word $\bff$, we find a MAP of difference $d=3$ and length $\ell=5$ given by
\[
f_{20}=f_{23}=f_{26}=f_{29}=f_{32}=\mathtt{0}.
\]
It can be shown that there is no MAP of difference $d=3$ longer than $5$ and so $A(3) = 5$ (see Section \ref{fixdalgo}).
As $f_{20}$ is the earliest starting position of a length-$5$ MAP of difference $3$ (all earlier MAPs are shorter), we have $i(3) = 20$.
\end{example}
Recall that the set of all infinite words is denoted by $\mc{A}^{\mathbb{N}_0}$. A word $\bfx$ is called {\em periodic} if there exists $x \in \mc{A}^*$ such that $\bfx=x^\omega$. A word $\bfx$ is called {\em eventually periodic} if there exist $x,y \in \mc{A}^*$ such that $\bfx=yx^\omega$. A word $\bfx$ is called {\em non-periodic} otherwise. A first simple observation provides a uniform lower bound on $A_\bfx(d)$ when $\bfx$ is a non-periodic sequence over a binary alphabet.
\begin{lem}\label{adgeq2}
Let \(\bfx\) be a non-periodic sequence over the alphabet $\{\mathtt{0,1}\}$. For every $d$, $A_{\bfx}(d)\geq2$.
\end{lem}
\begin{proof}
Since the sequence is non-empty, it is clear that $A_\bfx(d)$ is never $0$. Suppose that for some $d \geq 1$, we have $A_\bfx(d)=1$. Then for all \(k \geq 0\), we have $x_k \neq x_{k+d}$. Let $\bfx_{[0,d-1]} = x_0x_1 \cdots x_{d-1}$ be the length-$d$ prefix of $\bfx$. Let $\overline{x_i}$ denote the complement of $x_i$; i.e., $x_i=\mathtt{0} \iff \overline{x}_i = \mathtt{1}$. Given the observation that no symbol repeats with difference \(d\), we conclude that \(\bfx\) must have the following structure
\[\bfx = x_0x_1\cdots x_{d-1}\overline{x_0x_1\cdots x_{d-1}} x_0x_1\cdots x_{d-1}\overline{x_0x_1\cdots x_{d-1}}\cdots.\]
That is, \(\bfx = (x_0x_1\cdots x_{d-1}\overline{x_0x_1\cdots x_{d-1}})^\omega\). So, $\bfx$ is periodic with period \(2d\). This contradicts the non-periodicity of \(\bfx\). So, \(A_\bfx(d)=1\) must be false. That means, $A_\bfx(d) \geq 2$ for all $d$.
\end{proof}

\subsection{Automata Theory}\label{SEC:aut}
A {\em deterministic finite automaton} (DFA) is a finite-state machine that accepts or rejects an input string of symbols. A standard reference for the study of automatic sequences (including the following technical definitions) is the book of Allouche and Shallit \cite{AS2003}.
\begin{definition} \label{dfadef}
A {\em DFA} $M$ is defined as a 5-tuple $M = (\mathcal{Q},\Sigma, \delta, q_0, F)$, where 
\begin{itemize}
    \item $\mathcal{Q}$ is a finite set of states,
    \item $\Sigma$ is the finite input alphabet,
    \item $\delta\colon\mathcal{Q}\times\Sigma \mapsto \mathcal{Q}$ is the transition function,
    \item $q_0\in \mathcal{Q}$ is the initial state, and
    \item $F\subseteq \mathcal{Q}$ is the set of accepting states.
\end{itemize}
\end{definition}
We can extend $\delta$ iteratively to a function $\delta \colon \mathcal{Q} \times \Sigma^* \to \mathcal{Q}$ by defining, for a letter $a$ and finite word $u$, $\delta(q,au)\coloneq \delta(\delta(q,a),u)$, where $\delta(q,\varepsilon) = q$. Note that this definition implicitly assumes that we read words from left to right, which will correspond to reading representations of natural numbers from the {\em most significant digit} (msd-first). See Figure \ref{FIG:tmad4} for an example of a DFA where the input is read in binary with msd-first. We will almost always use msd-first throughout this text, although one can always define a DFA that reads inputs from right to left, which we refer to as reading from the {\em least significant digit} (lsd-first).

\subsubsection{\texorpdfstring{$k$}{k}-automatic sequences}
\begin{definition}{\label{dfaodef}}
Let $k\geq2$. A {\em $k$-DFAO} is a DFA with an output and whose input alphabet comprises $k$ symbols. That is, a $k$-DFAO is a $6$-tuple  $M = (\mathcal{Q}, \Sigma=\{0,1,\dots,k-1\}, \delta, q_0, \A, \vartheta)$ with $\mathcal{Q},\Sigma, \delta, q_0$ as given in Definition \ref{dfadef}, $\A$ is the output alphabet, and $\vartheta$ is a map $\vartheta\colon \mathcal{Q} \to \A$ called a {\em coding}. 

A sequence $\bfx$ is {\em $k$-automatic} if it is the output of a $k$-DFAO. That is, the $n$-th term in $\bfx$ is given by $x_n= \vartheta(\delta(q_0,(n)_k))$, where $(n)_k$ is the $k$-ary representation of the natural number $n$ seen as a word over $\{0,1,\ldots, k-1\}$.
\end{definition}
\begin{example}
The Thue--Morse word $\bft$ is the output of the $2$-DFAO shown in Figure \ref{FIG:thuemorsedfao} and is therefore $2$-automatic.
\begin{figure}
\centering
\begin{tikzpicture}
\tikzset{->,>=stealth',shorten >=1pt,node distance=2.5cm,every state/.style={thick, fill=white},initial text=$ $}
\node[state, initial] (q0) {$\mathtt{0}$};
\node[state, right of=q0] (q1) {$\mathtt{1}$};

\draw 
(q0) edge[loop above, looseness=12] node{0} (q0)
(q0) edge[bend right, below] node{1} (q1)
(q1) edge[bend right, above] node{1} (q0)
(q1) edge[loop above, looseness=12] node{0} (q1)
;
\end{tikzpicture}

\caption{A $2$-DFAO that generates the Thue--Morse word $\mathtt{0110100110010110\cdots}$.}
\label{FIG:thuemorsedfao}
\end{figure}
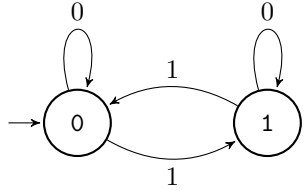
Comparing to Definition \ref{dfaodef}, the $2$-DFAO that generates $\bft$ is defined with $\mathcal{Q}=\mc{A}=\{\mathtt{0,1}\}$, $\Sigma=\{0,1\}$, $q_0=\mathtt{0}$, and the identity map $\vartheta$. For example, to find $t_6$, the $6$-th term in $\bft$, one reads $(6)_2=110$ with msd-first and records the output state as follows: \[\rightarrow\mathtt{0}\overset{1}{\longrightarrow}\mathtt{1}\overset{1}{\longrightarrow}\mathtt{0}\overset{0}{\longrightarrow}\mathtt{0}.\] So, $t_6=\mathtt{0}$.

Note that the Thue--Morse word is also the fixed point of the substitution $\mathtt{0}\mapsto \mathtt{01},\, \mathtt{1} \mapsto \mathtt{10}$.
\end{example} 

\subsection{\walnut}
We use the automatic theorem-proving software \walnut{}, originally written by Hamoon Mousavi \cite{mous_walnut}.
\walnut{} is based on an algorithm that can decidably prove or disprove first-order logical statements involving automatic sequences, made possible by B\"uchi's decidability theorem (\cite[Theorem 6.3.2]{shallit2022}).

For a detailed account of the development of \walnut{} and an introduction to its use, we recommend the book by Shallit \cite{shallit2022}. It is also worth mentioning that Shallit currently maintains a directory of academic works that have made use of \walnut{} in a research setting \cite{shallitwalnutwebpage}.

In this work, we make reference in several places to outputs generated by \walnut{}. In these cases, we provide code whose output can be confirmed by running the same command in the \walnut{} software.
Thankfully, \walnut{} code is intuitive to read for a working mathematician, and so for the benefit of the reader that is not familiar with \walnut{} but would like to understand the code provided in this work without needing to run it in the software, we provide some of the basics of \walnut{} syntax that are useful for understanding our code:
\begin{itemize}
    \item {\ttfamily E} is the existential quantifier, and {\ttfamily A} is the universal quantifier.
    \item {\ttfamily ?msd\_k} tells \walnut{} to represent integers in base-$k$ to be read msd-first. For example, {\ttfamily ?msd\_2} for binary and {\ttfamily ?msd\_fib} for base Fibonacci. In the absence of that, it represents the integers in binary by default.
    \item {\ttfamily \textasciitilde} is logical NOT, {\ttfamily \&} is logical AND, {\ttfamily |} is logical OR, {\ttfamily =>} is logical implication.
    \item {\ttfamily def} defines an automaton accepting the values of the free variables, making the logical statement true.
    \item {\ttfamily eval} returns \true \ (resp.~\false) when the output DFA accepts (resp.~rejects) all input values. 
    \item {\ttfamily T}, {\ttfamily RS}, and {\ttfamily F} are predefined in \walnut{} to be the Thue--Morse, Rudin--Shapiro, and Fibonacci words respectively.
	\item Representations of families of integers can be defined in {\walnut} using {\em regular expressions} using the {\ttfamily reg} command. For example, the following command defines a regular expression that accepts only values of the form $2^n$ for some $n\in\N$:
{\ttfamily reg ispower2 ?msd\_2 "0*10*";}.

Here, {\ttfamily u*} is to be interpreted as an arbitrarily long (possibly empty) concatenated string of {\ttfamily u}'s. In terms of automata, {\ttfamily u*} corresponds to a cycle whose labelled edges read {\ttfamily u}.
\end{itemize}

\section{Automatically proved results for MAPs}\label{SEC:results-walnut}
Most of the results in this section are proved using the automatic theorem-proving software \walnut{}.

\subsection{Encoding MAPs in \walnut{}}\label{fixdalgo}
We define the following basic \walnut{} commands that will be used to find the values of $A(d)$ and $i(d)$ for any fixed value of $d$ in an automatic sequence $\bfx$.  Note that, since the multiplication of two variables is prohibited in extended Presburger arithmetic, we must manually fix the difference $d$ when evaluating terms such as {\ttfamily d*k} in the following list of commands. Hence, the commands must be rerun for each new value of $d$.
\begin{itemize}
\item {\ttfamily def map "Ak (k<l) => X[i]=X[i+d*k]":
\\\# Accepts all pairs $(l,i)$ for which there is a MAP of length $l$ of difference $d$,\\ beginning at position $i$ in $\bfx$.}
\item {\ttfamily def imap "Ei \$map(i,l)":\\\# Accepts all values of $l$ for which there is a MAP of length $l$ of difference $d$.}
\item {\ttfamily def lmap "\$imap(l) \& Am (m>l) => \textasciitilde\$imap(m)":\\\# Accepts $A(d)$, the length of the longest MAP of difference $d$.}
\item {\ttfamily def fmap "\$map(i,l) \& Aj (j<i) => \textasciitilde\$map(j,l)":\\\# Accepts $i(d)$, the starting position of the first MAP of length $A(d)$; note that\\ {\ttfamily l} in {\ttfamily fmap} must be manually replaced with the value of $A(d)$ found by {\ttfamily lmap}.}
\end{itemize}

\begin{rem}
    For a fixed difference $d$, it is known to be decidable whether an automatic sequence admits an infinite MAP \cite[Proposition 22]{dg2019}. One may use the following command in {\walnut} to check the same: {\ttfamily eval infmapcheck "Ei (An (n>0) => X[i]=X[i+n*d])";}.
\end{rem}
\subsection{The Thue--Morse word}\label{misc}
Morgenbesser et al.\ \cite{mss2011} and Aedo et al.\ \cite{aedo2022} investigated bounds on the lengths of MAPs for the Thue--Morse word $\bft$. In particular, Aedo et al.\ showed that for all $d \geq 1$ and $n \geq 1$, $A(2^nd) = A(d)$. In particular, $A(2^n)=2$ for all $n \geq 1$, and this can be checked in \walnut{} (we leave the details to the interested reader). They were also able to prove the following uniform bound for all other values of $A(d)$. 
\begin{lem}[{\cite[Lemma 13]{aedo2022}}]\label{LEM:tmadgeq3}
Let $d > 1$ be an odd integer. Then $A(d)\geq 3$.
\end{lem}
The authors conjectured that this bound can be improved to $A(d)\geq 4$. That is, they claimed that there exists no odd $d$ for which $A(d)=3$. We resolve this conjecture and, as a consequence, find the optimal lower bound for $A(d)$.
\begin{prop}\label{tmadnot3}
Let $d\geq 1$.
If $d=2^n$ for $n \geq 0$, then $A(d)=2$.
If $d$ has an odd factor, then $A(d)\geq 4$. Further, there exists $d$ for which $A(d)=4$.
\end{prop}
\begin{proof}
The statement for powers of $2$ follows from the fact that $A(1)=2$ and $A(2^nd)=A(d)$.
For the second statement, we first prove that if $d>1$ is odd, then $A(d)\geq 4$. Running the following command in \walnut{} returns {\ttfamily TRUE}:
\begin{itemize}
    \item {\ttfamily def odd "En (d=2*n+1)";\\ \# Accepts all odd values of $d$.}
    \item {\ttfamily eval tmad\_geq4 "Ad (\$odd(d) \& d>1) => Ei (T[i]=T[i+d] \& T[i]=T[i+2*d] \& \\T[i]=T[i+3*d])";\\\# This checks if for all odd $d>1$, the MAP has length at least $4$.}
    \end{itemize}
    As $A(2^nd)=A(d)$, this confirms the same result for all $d$ with an odd factor.
To see that the bound is sharp, we observe that $A(11)=4$, as the input $(11)_2$ reaches an accepting state in the DFA given in Figure \ref{FIG:tmad4}.
\end{proof}

As $A(2^nd)=A(d)$ for all $d \geq 1$, it is only necessary to understand $A(d)$ for odd values of $d$. It would be convenient if a similar result held for $i(d)$. For instance, we have that $i(2^n)=2^n$ for all $n\geq 0$. This can be checked in \walnut{} or proved by elementary methods (see Proposition \ref{Prop:i2deqid}).
%
In principal, one may also verify in \walnut{} that for all $n \geq 0$, $i(2^nd)=2^ni(d)$ for any given fixed difference $d \geq 1$. The reason that this statement can be encoded in \walnut{} is because $A(2^nd)$ is fixed over all $n \geq 0$. Unfortunately, \walnut{} cannot be used (to our knowledge) to show that this relationship holds simultaneously for all $d$. Thankfully, we are able to show this by hand using traditional methods.
\begin{prop}\label{Prop:i2deqid}
For all $d \geq 1$ and for all $n \geq 0$, $i(2^nd)=2^ni(d)$. In particular, if $d$ is even, then $i(d)$ is even.
\end{prop}
\begin{proof}
We prove the statement for $n=1$, from which the full statement follows by a simple induction.

First, because $A(2d)=A(d)$, and the Thue--Morse substitution has constant substitution length $\ell=2$, by simply applying the substitution to the earliest MAP of length $d$, we get that $i(2d)\leq 2i(d)$. It now remains to show that $i(2d)\geq 2i(d)$.
 
For that, we use another defining property of $\bft$ that for all $n$, $t_{2n} = t_n$ and $t_{2n+1} = \overline{t_n}$. Suppose that $i(2d)$ is even.
    Then as
    \[t_{i(2d)} = t_{i(2d)+2d} = \cdots = t_{i(2d)+(A(d)-1)2d},\]
    we must also have
    \[t_{i(2d)/2} = t_{i(2d)/2+d} = \cdots = t_{i(2d)/2+(A(d)-1)d}.\]
    Hence $i(2d)/2$ is the first position of a MAP of difference $d$ and length $A(2d) = A(d)$.
    So, $i(d) \leq i(2d)/2$, giving $2i(d) \leq i(2d)$.

    Now suppose that $i(2d)$ is odd.
    Then as
    \[t_{i(2d)} = t_{i(2d)+2d} = \cdots = t_{i(2d)+(A(d)-1)2d},\]
    we must also have
    \[\overline{t_{(i(2d)-1)/2}} = \overline{t_{(i(2d)-1)/2+d}} = \cdots = \overline{t_{(i(2d)-1)/2+(A(d)-1)d}}.\]
    Hence $(i(2d)-1)/2$ is the first position of a MAP of difference $d$ and length $A(2d) = A(d)$.
    So, $i(d) \leq (i(2d)-1)/2$, giving $2i(d) \leq 2i(d)+1 \leq i(2d)$.
\end{proof}
\begin{rem}The above proof can easily be adapted to show that for any fixed point of a bijective substitution of length $\ell$, we have that for all $d \geq 1$ and $n \geq 0$, $i(\ell^nd)=\ell^ni(d)$. Here, by \emph{bijective substitution}, we mean a substitution $\theta$ of constant length $\ell$ such that for each $0 \leq i \leq \ell-1$ there exists a bijection $\alpha_i \colon \mc A \to \mc A$ such that for all $a \in \mc A$, $\theta(a)_i = \alpha_i(a)$.\end{rem}

\subsubsection{Preimages of \texorpdfstring{$A(d)$}{A(d)}} Though Proposition \ref{Prop:i2deqid} is proved using traditional methods, in the initial attempt to use {\walnut} to prove it, an interesting fact emerged which we mention below.
\begin{prop}\label{PROP:lod35a4}
    The largest odd $d$ for which $A(d)=4$ is $35$.
\end{prop}
\begin{proof}
    The proof is enumerating the DFA generated by the following \walnut{} command:\\
    {\ttfamily def tmad4 "Ei (T[i]=T[i+d] \& T[i]=T[i+2*d] \& T[i]=T[i+3*d] \& \\T[i]!=T[i+4*d]) \& (Aj (j>=0) => \textasciitilde(T[j]=T[j+d] \& T[j]=T[j+2*d] \& \\ T[j]=T[j+3*d] \& T[j]=T[j+4*d]))";\\
    This outputs the DFA given in Figure \ref{FIG:tmad4}.}
    
    \begin{figure}[ht]
        \centering
        \includegraphics[scale=0.05]{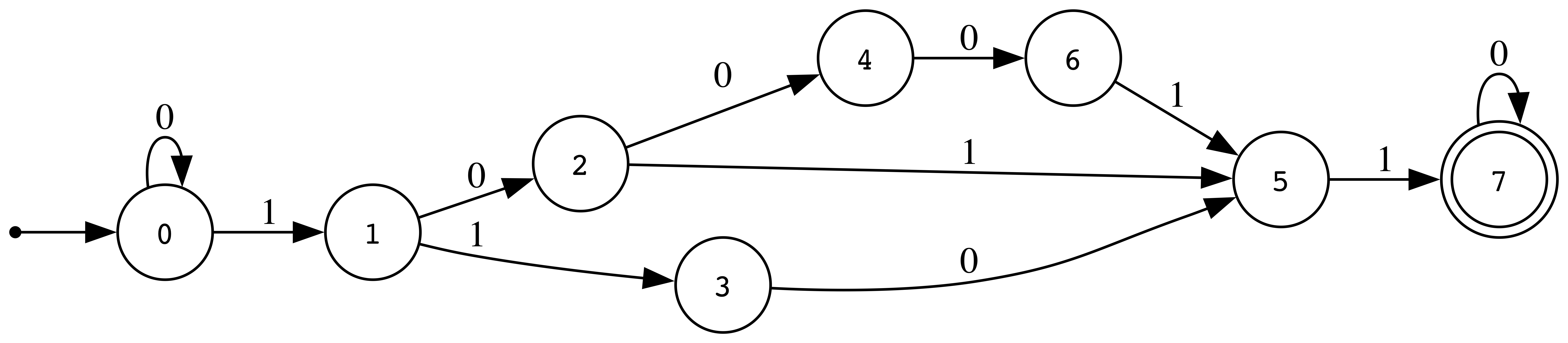}
        \caption{DFA enumerating all values of $d$ for which $A(d)=4$.}
        \label{FIG:tmad4}
    \end{figure}
    The (unique) longest accepted input ending in a {\ttfamily 1} is {\ttfamily 100011}, corresponding to $35$.
\end{proof}
Let $O_{\max} \colon \N \to \N \cup \{-\infty, \infty\}$ denote the maximal odd value of $d$ for which $A(d)=n$. That is,
$O_{\max}(n) \coloneqq \max \{d \mid A(d) = n,\ d \text{ odd}\}$, where it is understood that $O_{\max}(n) \coloneqq -\infty$ if the set of odd $d$ for which $A(d) = n$ is empty, and $O_{\max}(n) \coloneqq \infty$ if there are infinitely many odd $d$ for which $A(d) = n$. Hence, for instance, Proposition \ref{PROP:lod35a4} can be restated as $O_{\max}(4) = 35$. We similarly calculate the following values of $O_{\max}(n)$ in Table \ref{tab:tmlodforad} by suitably modifying the \walnut{} command in the proof of Proposition \ref{PROP:lod35a4}.

\begin{table}[ht]
    \centering
    \begin{tabular}{|c|c|c|c|c|c|c|c|}
    \hline
       $n$  & $1$ & $2$& $3$& $4$& $5$& $6$& $7$ \\
       \hline
       $O_{\max}(n)$ & $-\infty$& $1$& $-\infty$& $35$ & $29$& $\infty$& $\infty$\\
        \hline
    \end{tabular}
    \caption{$n:=A(d)$, and $O_{\max}(n):=$ largest odd $d$ for which $A(d)=n$.}
    \label{tab:tmlodforad}
\end{table}
    The time taken by \walnut{} to find the DFA enumerating $O_{\max}$ grows exponentially. For $n=1,\dots ,5$, it took $3,11,86,1926,21746$ milliseconds respectively. For $n=6$, \walnut{} took approximately 10 minutes. For $n=7$, \walnut{} took approximately an hour and over $300$GB of memory. On hardware available to the authors, \walnut{} was unable to calculate $O_{\max}$ for $n \geq 8$ due to memory limitations. We do not understand the behaviour of the function $O_{\max}$ from the values in Table \ref{tab:tmlodforad}. Indeed, it is unclear if $O_{\max}$ takes mostly finite or infinite values.
This motivates the following set of questions.
\begin{que}
   Are there infinitely many values of $n$ for which $O_{\max}(n) = -\infty$, or is $\{1,3\}$ the full set of values for which $O_{\max} = -\infty$?
   Are there infinitely many values of $n$ for which $O_{\max}(n)=\infty$?
   Are there infinitely many values of $n$ for which $O_{\max}(n)<\infty$, or is $\{2,4,5\}$ the full set of values for which $O_{\max}$ is finite?
\end{que}
\subsubsection{\texorpdfstring{$i(d)$}{i(d)} for particular families of \texorpdfstring{$d$}{d}}
We cannot produce an automaton whose output is $i(d)$ for all $d \geq 1$ simultaneously, as $A(d)$ varies in $d$ and this would therefore require a product of variables in \walnut{}. Likewise, for certain infinite families of $d$, we may conjecture that $i(d)$ takes a particular form, but this cannot be checked if $A(d)$ is not constant on the family. Some conjectures of this form are given below:

Aedo et al.\ showed that for all $n\geq 2$, $A(2^n+1)=2^n+2$ and for all $n\geq 2$, $A(2^n-1)= 2^n+4$ if $n$ is even and $A(2^n-1)=2^n$ if $n$ is odd \cite{aedo2022}. Further to this, based on empirical observations (see \seqnum{A342827}), it would appear that the identities given in Conjecture \ref{CONJ:tmid} hold, but a proof remains elusive.
\begin{conj}\label{CONJ:tmid}
We have
\begin{align*}
    i(2^n+1)&=3.2^{2n}-2^n-1,\\
    i(2^{2n}-1)&=3.2^{4n}-2^{2n}+1,\\
    i(2^{2n+1}-1) &= 2^{2n+1}-1.
\end{align*}
\end{conj}
There are still many other open questions relating to MAPs in the Thue--Morse word. However, as this is not the focus of this work, we address an automatic sequence of interest, the Rudin--Shapiro word, before finally focussing our attention on our principal sequence of interest, the Fibonacci word.

\subsection{The Rudin--Shapiro word}\label{rudinshapiro}
The {\em Rudin--Shapiro} sequence $\bfr:= \mathtt{0001001000011101}\cdots$ 
is another binary automatic sequence of frequent interest \cite{golay1949,shapiro52,rudin59}. It can be defined as a coding of the fixed point of the substitution $\mathtt{a}\mapsto \mathtt{ab},\,\mathtt{b}\mapsto \mathtt{ac},\,\mathtt{c}\mapsto \mathtt{db},\,\mathtt{d}\mapsto \mathtt{dc}$, with coding morphism given by $\mathtt{a},\mathtt{b}\mapsto \mathtt{0}, \text{ and } \mathtt{c},\mathtt{d}\mapsto \mathtt{1}$.
Aedo et al. studied MAPs of generalisations of the Rudin--Shapiro word \cite{aedo2023}. In particular, they showed that there is no infinite MAP in $\bfr$ (\cite[Proposition 41]{aedo2023}). They also proved that $A(2^nd)=A(d)$; in particular $A(2^n)=A(1)=4$ (\cite[Section 4.1]{aedo2023}).
Sobolewski extended these results by considering families of values of $d$ of the form $2^n\pm k$ for some $k$ \cite{sobo23}.
Due to the aforementioned limitations of \walnut{}, we are only able to reprove some of these results by automatic methods; namely when $A(d)$ is constant. In such cases, we are often able to determine the first positions of longest MAPs, $i(d)$.

\begin{prop}
    For all $n\geq0$, $i(2^n)=7\cdot 2^n$.
\end{prop}
\begin{proof}
We first confirm that for all $n \geq 0$, $A(2^n)=4$. It is then only necessary to check that there is a MAP of length $4$ and difference $2^n$ beginning at position $7\cdot 2^n$ and no earlier such MAPs appear.
\begin{itemize}
   \item {\ttfamily def run4rs "RS[i]=RS[i+d] \& RS[i]=RS[i+2*d] \& RS[i]=RS[i+3*d]";\\\# Accepts all pairs $(d,i)$ for which $i$ is the starting position of a MAP of\\ length at least $4$ with difference $d$.}
   \item {\ttfamily def run5rs "RS[i]=RS[i+d] \& RS[i]=RS[i+2*d] \& RS[i]=RS[i+3*d] \&\\ RS[i]=RS[i+4*d]";\\\# Accepts all $i$ for which $i$ is the starting position of a MAP of length at\\ least $5$ with difference $d$.}
   \item {\ttfamily eval rsad4check "Ad (\$power2(d)) => Ei (\$run4rs(d,i) \& Aj \textasciitilde\$run5rs(d,j))";\\\# Returns \true{} if and only if for all $n\geq 0$, $A(2^n)=4$.\\ \walnut{} returns \true{}.}
   \item {\ttfamily eval propcheck "Ad (\$ispower2(d)) => \\(\$run4rs(d,7*d) \& (Aj (j<7*d) => \textasciitilde\$run4rs(d,j)))";\\\# Returns \true{} if and only if for all $n \geq 0$, there is a MAP of length $4$ and\\ difference $2^n$ beginning at position $7\cdot 2^n$ and no earlier such MAP appears.\\ \walnut{} returns \true. }
\end{itemize}
\end{proof}
\subsection{The Fibonacci word}
We now shift our attention to the Fibonacci word, the fixed point of the substitution $\mathtt{0}\mapsto \mathtt{01},\, \mathtt{1}\mapsto \mathtt{0}$. As this substitution is not constant-length, and the Fibonacci word cannot be encoded from the fixed point of a constant-length substitution, we need a more general setting in which to consider this sequence automatic.
\subsubsection{Fibonacci-automatic}
Let $\{F_0,F_1,\ldots\}$ denote the set of Fibonacci numbers defined by $F_0=0$, $F_1=1$ and recursively $F_{n+2}=F_{n+1}+F_{n}$ for all $n\geq0$. Ostrowski \cite{ostro1922}, Lekkerkerker \cite{lekker1951} and Zeckendorf \cite{zecken1972} showed that there is a unique way to represent non-negative integers as the sum of distinct Fibonacci numbers with the restriction that no pair of consecutive Fibonacci numbers are used. This is often called the {\em Zeckendorf representation} or {\em base Fibonacci}.
So, every non-negative integer $n$ has a unique representation of the form
\[n=\sum_{i \geq 2} a_iF_i,\]
where $a_i\in\{0, 1\}$, and $a_ia_{i+1} = 0$ for all $i\geq 2$ \cite{csh1972}. This representation is conventionally written as a binary word $(n)_F$ with msd-first: $(n)_F \coloneqq a_ta_{t-1}\cdots a_2$, with $(0)_F = \varepsilon$, the empty word. For example, $(45)_F=10010100$ as $45=F_9+F_6+F_4$. 
We say an {\em $F$-DFAO} is a DFAO with input read in Zeckendorf representation, and a sequence is {\em Fibonacci-automatic} if and only if an $F$-DFAO generates it (see, e.g., \cite[Section 5.9]{shallit2022}).

\begin{example}\label{fibdfadef}
    The $n$-th term of the Fibonacci word $\bff$ can also be given by the last digit in the Zeckendorf representations of $n$. Therefore, $\bff$ is an output sequence of the $F$-DFAO shown in Figure \ref{fibdfa_fig}, which has states and output alphabet $\mathcal{Q} =\mc A = \{\mathtt{0},\mathtt{1}\}$ with initial state $q_0 = \mathtt{0}$, the input alphabet $\Sigma = \{0,1\}$, with the identity coding map $\tau\colon \mathcal{Q} \to \mathcal{A}$.

For example, in order to determine the $11$-th term in $\bff$, we first find the Zeckendorf representation of $11$ by noting that $11 = 8+3 = F_5+F_3$, so $(11)_F = 10100$. We begin at the initial state $q_0 = \mathtt{0}$ and follow the sequence of edges that reads $10100$ from the most significant digit to the least. So we follow the path $\rightarrow\mathtt{0}\stackrel{1}{\longrightarrow} \mathtt{1} \stackrel{0}{\longrightarrow} \mathtt{0} \stackrel{1}{\longrightarrow}\mathtt{1} \stackrel{0}{\longrightarrow} \mathtt{0}\stackrel{0}{\longrightarrow} \mathtt{0}$ which confirms $f_{11} = \mathtt{0}$. In this simple DFAO, the output is $i$ if and only if the input ends in $i$, but this need not be the case for more complex DFAOs.
\end{example}
\begin{figure}[ht]
\centering
\begin{tikzpicture}
\tikzset{->,>=stealth',shorten >=1pt,node distance=2.5cm,every state/.style={thick, fill=white},initial text=$ $}
\node[state,initial] (q0) {$\mathtt{0}$};
\node[state, right of=q0] (q1) {$\mathtt{1}$};
\draw 
(q0) edge[loop above] node{0} (q0)
(q0) edge[bend left, above] node{1} (q1)
(q1) edge[bend left, below] node{0} (q0)
;
\end{tikzpicture}
\caption{$F$-DFAO generating $\bff$}
\label{fibdfa_fig}
\end{figure}
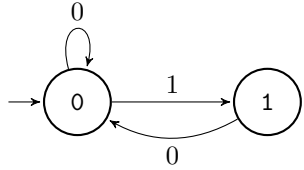

\subsubsection{Investigating MAPs in the {F}ibonacci word with {\walnut}}\label{SEC:fibwalnutres}
By running the commands {\ttfamily map}, {\ttfamily imap}, {\ttfamily lmap} and {\ttfamily fmap} (from Subsection \ref{fixdalgo}) for the Fibonacci word with fixed values of the difference $d$, we found the values of $A(d)$ and $i(d)$ for $1 \leq d \leq 234$ as listed in Table \ref{tab:fibdadid}. The $A(d)$ values are listed as sequence \seqnum{A339949} on the OEIS \cite{oeis}. Note that when encoding propositions for the  Fibonacci word, {\ttfamily ?msd\_fib} is recalled each time as the necessary numeration system of choice.

\begin{table}[ht]
\centering{
\scalebox{0.84}{
\begin{tabular}{|r|r|r||r|r|r||r|r|r||r|r|r||r|r|r||r|r|r|}
\hline
{\bf{$d$}} & {$A(d)$} & {$i(d)$} & {\bf{$d$}} & {$A(d)$} & {$i(d)$} & {\bf{$d$}} & {$A(d)$} & {$i(d)$} & {\bf{$d$}} & {$A(d)$} & {$i(d)$} & {\bf{$d$}} & {$A(d)$} & {$i(d)$} & {\bf{$d$}} & {$A(d)$} & {$i(d)$}\\
    \hline
    \bf{1} & 2     & 2     & \bf{40} & 3     & 7     & \bf{79} & 4     & 7     & \bf{118} & 9     & 20    & \bf{157} & 20    & 32    & \bf{196} & 5     & 11\\
    \bf{2} & 3     & 3     & \bf{41} & 2     & 0     & \bf{80} & 4     & 16    & \bf{119} & 4     & 2     & \bf{158} & 2     & 2     & \bf{197} & 3     & 7 \\
    \bf{3} & 5     & 20    & \bf{42} & 15    & 20    & \bf{81} & 11    & 87    & \bf{120} & 4     & 3     & \bf{159} & 3     & 11    & \bf{198} & 2     & 0 \\
    \bf{4} & 6     & 16    & \bf{43} & 4     & 23    & \bf{82} & 2     & 2     & \bf{121} & 3     & 2     & \bf{160} & 6     & 20    & \bf{199} & 56    & 28656\\
    \bf{5} & 7     & 11    & \bf{44} & 4     & 11    & \bf{83} & 3     & 32    & \bf{122} & 4     & 21    & \bf{161} & 36    & 70    & \bf{200} & 4     & 31 \\
    \bf{6} & 3     & 20    & \bf{45} & 4     & 20    & \bf{84} & 8     & 20    & \bf{123} & 34    & 32    & \bf{162} & 6     & 87    & \bf{201} & 3     & 3  \\
    \bf{7} & 2     & 0     & \bf{46} & 4     & 8     & \bf{85} & 6     & 2     & \bf{124} & 2     & 2     & \bf{163} & 3     & 7     & \bf{202} & 4     & 2 \\
    \bf{8} & 12    & 143   & \bf{47} & 13    & 11    & \bf{86} & 5     & 32    & \bf{125} & 3     & 3     & \bf{164} & 2     & 0     & \bf{203} & 6     & 783 \\
    \bf{9} & 4     & 2     & \bf{48} & 2     & 2     & \bf{87} & 3     & 2     & \bf{126} & 5     & 7     & \bf{165} & 26    & 88    & \bf{204} & 8     & 11  \\
    \bf{10} & 4     & 11    & \bf{49} & 3     & 11    & \bf{88} & 4     & 55    & \bf{127} & 14    & 16    & \bf{166} & 4     & 10    & \bf{205} & 3     & 54  \\
    \bf{11} & 4     & 54    & \bf{50} & 7     & 20    & \bf{89} & 123   & 231   & \bf{128} & 6     & 11    & \bf{167} & 3     & 3     & \bf{206} & 2     & 0 \\
    \bf{12} & 4     & 8     & \bf{51} & 8     & 36    & \bf{90} & 2     & 2     & \bf{129} & 3     & 7     & \bf{168} & 4     & 7     &\bf{207} & 10    & 54\\
    \bf{13} & 18    & 32    & \bf{52} & 5     & 11    & \bf{91} & 3     & 3     & \bf{130} & 2     & 0     & \bf{169} & 4     & 16    & \bf{208} & 4     & 2\\
    \bf{14} & 2     & 2     & \bf{53} & 3     & 7     & \bf{92} & 5     & 20    & \bf{131} & 17    & 88    & \bf{170} & 10    & 32    & \bf{209} & 4     & 3 \\
    \bf{15} & 3     & 11    & \bf{54} & 2     & 0     & \bf{93} & 8     & 50    & \bf{132} & 4     & 23    & \bf{171} & 2     & 2     & \bf{210} & 3     & 2 \\
    \bf{16} & 6     & 7     & \bf{55} & 77    & 6764  & \bf{94} & 7     & 11    & \bf{133} & 4     & 32    & \bf{172} & 3     & 32    & \bf{211} & 4     & 42 \\
    \bf{17} & 20    & 70    & \bf{56} & 4     & 31    & \bf{95} & 3     & 20    & \bf{134} & 4     & 7     & \bf{173} & 8     & 7     &  \bf{212} & 27    & 32  \\
    \bf{18} & 5     & 3     & \bf{57} & 3     & 3     & \bf{96} & 2     & 0     & \bf{135} & 4     & 8     & \bf{174} & 6     & 91    &   \bf{213} & 2     & 2\\
    \bf{19} & 3     & 7     & \bf{58} & 5     & 376   & \bf{97} & 13    & 54    & \bf{136} & 12    & 11    & \bf{175} & 4     & 3     & \bf{214} & 3     & 3\\
    \bf{20} & 2     & 0     & \bf{59} & 6     & 84    & \bf{98} & 4     & 2     & \bf{137} & 2     & 2     & \bf{176} & 3     & 2     & \bf{215} & 6     & 232 \\
    \bf{21} & 30    & 986   & \bf{60} & 8     & 11    & \bf{99} & 4     & 11    & \bf{138} & 3     & 11    & \bf{177} & 4     & 110   & \bf{216} & 28    & 160\\
    \bf{22} & 4     & 10    & \bf{61} & 3     & 54    & \bf{100} & 4     & 20    & \bf{139} & 7     & 7     & \bf{178} & 62    & 231   &  \bf{217} & 6     & 11\\
    \bf{23} & 3     & 3     & \bf{62} & 2     & 0     & \bf{101} & 4     & 8     & \bf{140} & 6     & 15    & \bf{179} & 2     & 2     &  \bf{218} & 3     & 7\\
    \bf{24} & 4     & 7     & \bf{63} & 10    & 20    & \bf{102} & 16    & 32    & \bf{141} & 5     & 11    & \bf{180} & 3     & 3     &  \bf{219} & 2 & 0 \\
    \bf{25} & 4     & 16    & \bf{64} & 4     & 2     & \bf{103} & 2     & 2     & \bf{142} & 3     & 7     & \bf{181} & 5     & 7     &\bf{220} & 20    & 10945  \\
    \bf{26} & 9     & 11    & \bf{65} & 4     & 3     & \bf{104} & 3     & 11    & \bf{143} & 2     & 0     & \bf{182} & 8     & 16    & \bf{221} & 4     & 10  \\
    \bf{27} & 2     & 2     & \bf{66} & 3     & 2     & \bf{105} & 6     & 7     & \bf{144} & 200   & 75024 & \bf{183} & 7     & 32    & \bf{222} & 4     & 87 \\
    \bf{28} & 3     & 87    & \bf{67} & 4     & 42    & \bf{106} & 12    & 70    & \bf{145} & 4     & 86    & \bf{184} & 3     & 20    &  \bf{223} & 4 & 7\\
    \bf{29} & 9     & 376   & \bf{68} & 24    & 87    & \bf{107} & 5     & 3     & \bf{146} & 3     & 3     & \bf{185} & 2     & 0     & \bf{224} & 4     & 16\\
    \bf{30} & 4     & 2     & \bf{69} & 2     & 2     & \bf{108} & 3     & 7     & \bf{147} & 5     & 20    & \bf{186} & 14    & 20    & \bf{225} & 11    & 11 \\
    \bf{31} & 4     & 3     & \bf{70} & 3     & 3     & \bf{109} & 2     & 0     & \bf{148} & 6     & 16    & \bf{187} & 4     & 2     & \bf{226} & 2     & 2\\
    \bf{32} & 3     & 2     & \bf{71} & 6     & 54    & \bf{110} & 39 & 10945 & \bf{149} & 8     & 87    & \bf{188} & 4     & 11    &  \bf{227} & 3 & 32\\
    \bf{33} & 4     & 21    & \bf{72} & 78    & 304   & \bf{111} & 4     & 65    & \bf{150} & 3     & 20    & \bf{189} & 4     & 20    & \bf{228} & 8     & 4180\\
    \bf{34} & 47    & 87    & \bf{73} & 6     & 11    & \bf{112} & 3     & 3     & \bf{151} & 2     & 0     & \bf{190} & 4     & 8     &  \bf{229} & 6 & 2 \\
    \bf{35} & 2     & 2     & \bf{74} & 3     & 7     & \bf{113} & 4     & 7     & \bf{152} & 11    & 20    & \bf{191} & 14    & 11    & \bf{230} & 5 & 32\\
    \bf{36} & 3     & 3     & \bf{75} & 2     & 0     & \bf{114} & 4     & 3     & \bf{153} & 4     & 2     & \bf{192} & 2     & 2     &   \bf{231} & 3     & 2\\
    \bf{37} & 5     & 7     & \bf{76} & 22    & 2583  & \bf{115} & 9     & 32    & \bf{154} & 4     & 11    & \bf{193} & 3     & 11    & \bf{232} & 4     & 233\\
    \bf{38} & 10    & 16    & \bf{77} & 4     & 11    & \bf{116} & 3     & 609   & \bf{155} & 4     & 143   & \bf{194} & 7     & 88  & \bf{233} & 322     & 608\\
    \bf{39} & 6     & 3     & \bf{78} & 3     & 3     & \bf{117} & 2     & 0     & \bf{156} & 4     & 13    & \bf{195} & 10    & 413   & \bf{234} & 2     & 2\\
    \hline
\end{tabular}
}
}
\caption{Values of $A(d)$, and $i(d)$ for the Fibonacci word calculated using \walnut{}.}
\label{tab:fibdadid}
\end{table}
The following simple Lemma \ref{LEM:idnever1} allows us to always consider $\mathtt{0}$  without loss of generality to be the repeated element in the longest MAP for the Fibonacci word.
\begin{lem}\label{LEM:idnever1}
In the Fibonacci word $\bff$, for all $d \geq 1$, $f_{i(d)} = \mathtt{0}$.
\end{lem}
\begin{proof}
Suppose $f_{i(d)}=\mathtt{1}$. Then $f_{i(d)+k\cdot d}=\mathtt{1}$ for all $0\leq k< A(d)$.
In $\bff$, all $\mathtt{1}$'s are isolated. So, each $\mathtt{1}$ is preceded by a $\mathtt{0}$. Hence, for all $0 \leq k<A(d)$, we have $f_{i(d)+k\cdot d-1}=\mathtt{0}$. But this means that $f_{i(d)-1}$ is the starting position of a MAP of length $A(d)$ with difference $d$. It follows that $f_{i(d)}$ cannot be $\mathtt{1}$.
\end{proof}

 \begin{prop}\label{propAd2i02}
For all $d$, if $A(d)=2$ then either $i(d)=0$ or $i(d)=2$.
\end{prop}
\begin{proof} We run the following commands in \walnut{}.
\begin{itemize}
\item{\ttfamily def fibad2 "?msd\_fib Ek (F[k]=F[k+d] \& F[k]!=F[k+2*d]) \& \\(Aj (j>=0) => \textasciitilde(F[j]=F[j+d] \& F[j]=F[j+2*d]))";
\\\# Accepts the values of $d$ for which $A(d)=2$.\\The $F$-DFAO output by this command is shown in Figure \ref{fibad2fig}.}
\item {\ttfamily eval fibad2fpos0or2 "?msd\_fib Ad (\$fibad2(d)) => (F[0]=F[d] | F[2]=F[2+d])";\\\# Checks that if $A(d)=2$, then either $i(d)=0$ or $i(d)=2$.\\\walnut{} returns \true.}
\end{itemize} 
\begin{figure}[ht]
\centering
\includegraphics[scale=0.04]{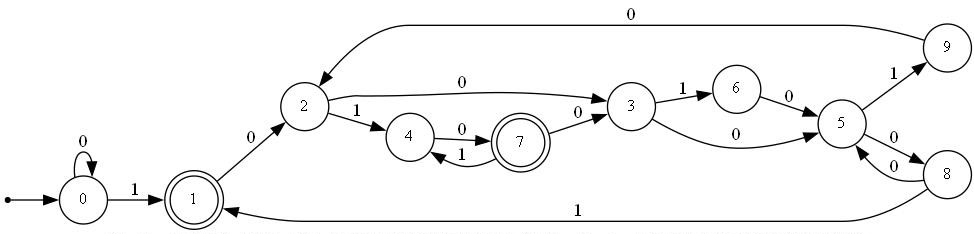}
\caption{DFA accepting only the values of $d$ for which $A(d)=2$.}
\label{fibad2fig}
\end{figure}
\end{proof}

\begin{rem} The DFA obtained from {\ttfamily fibad2} in Figure \ref{fibad2fig} is also mentioned in Shallit's book \cite[Figure 8.3]{shallit2022}, and enumerates sequence \seqnum{A339950} in the OEIS.
\end{rem}

The converse of a part of Proposition \ref{propAd2i02} that if $i(d)=2$ then $A(d)=2$ is \false \ (from Table \ref{tab:fibdadid}), because $i(9)=2$ but $A(9)=4\neq2$.
We claim that the converse of the other part of Proposition \ref{propAd2i02} holds. That is, for all $d$, if $i(d)=0$ then $A(d)=2$. But since $i(d)$ depends on $d$ and  $A(d)$, it cannot be encoded in \walnut. Fortunately, we are able to prove it in Proposition \ref{PROP:id0Ad2} using traditional dynamical methods.

\begin{prop}\label{PROP:ad3fib}
    For all $d$, $A(d) = 3$ if and only if $d$ is accepted by the automaton in Figure \ref{fig_Ad3fib}.
    \begin{figure}[ht]
    \includegraphics[scale=0.046]{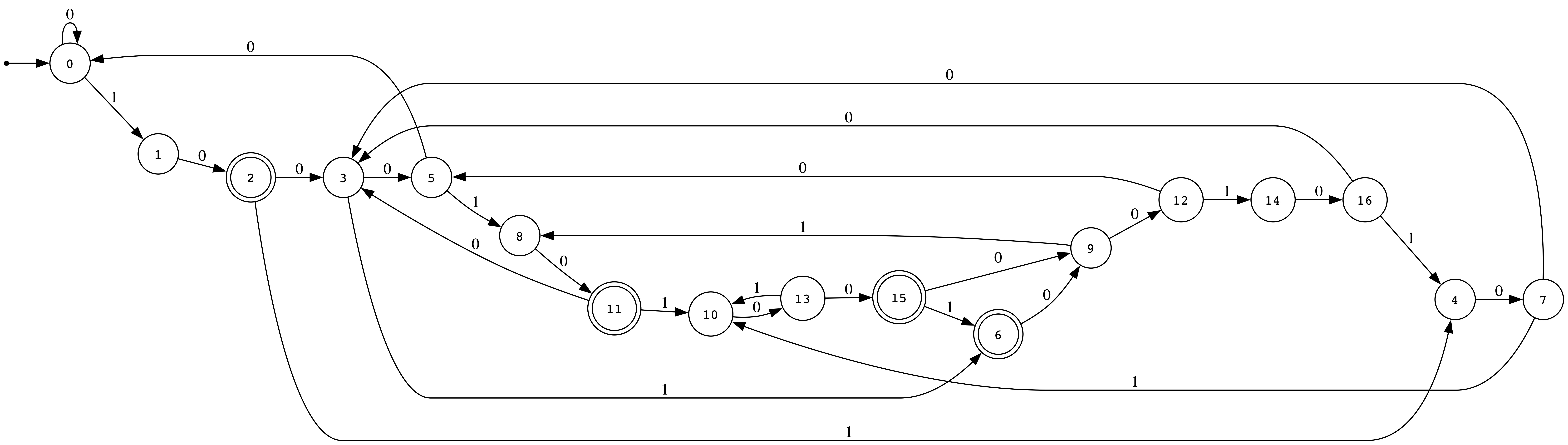}
    \caption{DFA accepting every $d$ for which $A(d)=3$.}
    \label{fig_Ad3fib}
    \end{figure}
\end{prop}
\begin{proof}
We run the following command in \walnut.
\begin{itemize}
\item {\ttfamily def fibad3aut "?msd\_fib Ek (F[k]=F[k+d] \& F[k]=F[k+2*d]) \& \\ (Aj \textasciitilde(F[j]=F[j+d] \& F[j]=F[j+2*d] \& F[j]=F[j+3*d]))";\\\# Accepts $d$ if $A(d)=3$.\\ The $F$-DFAO output by this command is shown in Figure \ref{fig_Ad3fib}.}
\end{itemize}
\end{proof}
Upon enumerating, we find the following sequence of values of $d$ for which $A(d)=3$:
\[ 2,6,15,19,23,28,32,36,40,49,\dots.\]
(See more values in Table \ref{tab:fibdadid}.)
As can be seen from the automaton, even though the sequence of accepted values for $A(d)=3$ is automatic, it is not simple to describe these values, say by an easy-to-understand regular expression. Still, the above suggests that we can similarly produce automata that enumerate all values of $d$ for which $A(d) = l$ for some fixed $l$. Indeed, this is made precise in Proposition \ref{PROP:gdranges}.

We now investigate the values of $A(d)$ and $i(d)$ for specific families of values of $d$.

Propositions \ref{PROP:dfn-1fib} and \ref{PROP:dfn+1fib} completely describe the values of $A(d)$ and $i(d)$ when $d=F_n-1$ and $d=F_n+1$ respectively.
\begin{prop}\label{PROP:dfn-1fib}
Let $d=F_n-1$ for $n\geq 3$. We have:
\[
A(d) =
\begin{cases}
2 &\text{ if } n=3,\\
3 &\text{ if } n=4,\\
6 &\text{ if } n=5,\\
2 &\text{ if } n \geq 6, \ n\text{ even}, \\
4 &\text{ if } n \geq 7, \ n\text{ odd}, 
\end{cases}
\qquad \text{ and } \qquad
i(d) =
\begin{cases}
2 &\text{ if } n=3,\\
3 &\text{ if } n=4,\\
16 &\text{ if } n=5,\\
0 &\text{ if } n \geq 6, \ n\text{ even},\\
F_{n-1} &\text{ if } n \geq 7, \ n\text{ odd}.
\end{cases}
\]
Moreover, we have a partial converse. That is, for all $d \geq 1$, the following hold:
\begin{itemize}
\item $A(d) = 2$ and $i(d)=0$ if and only if $d=F_n-1$ and $n\geq 6$ is even;
\item $A(d) = 4$ and $i(d)=F_{n-1}$ if and only if $d = F_n-1$ and $n\geq 7$ is odd.

\end{itemize}
\end{prop}
    \begin{proof}
    The values of $A(d)$ and $i(d)$ for $n=3,4,5$ can be checked by hand, or verified from the data presented in Table \ref{tab:fibdadid}. Recall {\ttfamily fibad2} from the proof of Proposition \ref{propAd2i02}.
    For all $n \geq 6$, we run the following \walnut{} commands:
\begin{itemize}
    \item {\ttfamily reg fiboyes msd\_fib "0*10*";\\\# Accepts all Fibonacci numbers.}
    \item {\ttfamily reg f\_nevenyes msd\_fib "0*1(00)*";\\\# Accepts all even-indexed Fibonacci numbers.}
    \item {\ttfamily reg f\_noddyes msd\_fib "0*10(00)*";\\\# Accepts all odd-indexed Fibonacci numbers.}
    \item {\ttfamily reg adjfibnodd msd\_fib msd\_fib "[0,0]*[1,0][0,1]([0,0][0,0])*";\\\# Accepts all pairs of consecutive Fibonacci numbers $(F_n,F_{n-1})$ for $n \geq 3$ odd.}
    \item {\ttfamily def fibad4fpos "?msd\_fib (F[i]=F[i+d] \& F[i]=F[i+2*d] \&\\ F[i]=F[i+3*d] \& F[i]!=F[i+4*d]) \& (Aj (j>=0) =>  \textasciitilde(F[j]=F[j+d] \&\\ F[j]=F[j+2*d] \& F[j]=F[j+3*d] \& F[j]=F[j+4*d])) \& (Ak (k<i) => \\\textasciitilde(F[k]=F[k+d] \& F[k]=F[k+2*d] \& F[k]=F[k+3*d]))";\\ 
    \# Accepts all pairs $(d,i)$ for which $A(d)=4$, and the starting position of the \\first MAP of length $4$ is $i$.}
	\item {\ttfamily eval prop\_neven "?msd\_fib Ad (\$fiboyes(d+1) \& d>=7) => \\(\$fibeven(d+1) <=> (\$fibad2(d) \& F[d]=F[0]))";\\ \# This checks that for all $d=F_n-1\geq 7$, $n$ is even if and only if $A(d)=2$ and $i(d)=0$.\\\walnut{} returns \true.}
        \item {\ttfamily eval prop\_nodd "?msd\_fib Ad (\$fiboyes(d+1) \& d>=12) => \\(\$fibodd(d+1) <=> Ei (\$fibad4fpos(d,i) \& \$adjfibnodd(d+1,i)))";\\\# Checks that for all $d=F_n-1\geq 12$, $n$ is odd if and only if $A(d)=4$ and $i(d)=F_{n-1}$.\\\walnut{} returns \true.}
\end{itemize}
    \end{proof}
\begin{prop}\label{PROP:dfn+1fib}
Let $d=F_n+1$ for $n\geq 0$. We have:
\[
A(d) =
\begin{cases}
2 &\text{ if } n=0,\\
3 &\text{ if } n=1,2,5,\\
5 &\text{ if } n=3,\\
6 &\text{ if } n=4,\\
4 &\text{ if } n \geq 6, \ n\text{ even}, \\
2 &\text{ if } n \geq 7, \ n\text{ odd}, 
\end{cases}
\qquad \text{ and } \qquad
i(d) =
\begin{cases}
2 &\text{ if } n=0,\\
3 &\text{ if } n=1,2,\\
20 &\text{ if } n=3,5,\\
16 &\text{ if } n=4,\\
F_{n-1}-3 &\text{ if } n \geq 6, \ n\text{ even}, \\
2 &\text{ if } n \geq 7, \ n\text{ odd}.
\end{cases}
\]
\end{prop}
The proof follows similarly to the proof of Proposition \ref{PROP:dfn-1fib} and so we leave this as a \walnut{} exercise for the interested reader.\hfill\qed

\begin{prop}\label{PROP:dfn+2fib}
Let $d=F_n+2$ for $n\geq 0$. We have:
\[
A(d) =
\begin{cases}
3 &\text{ if } n=0,\\
5 &\text{ if } n=1,2,\\
6 &\text{ if } n=3,\\
7 &\text{ if } n=4,\\
2 &\text{ if } n=5,\\
4 &\text{ if } n=6,\\
3 &\text{ if } n\geq7,
\end{cases}
\qquad \text{ and } \qquad
i(d) =
\begin{cases}
3  &\text{ if } n=0,\\
20 &\text{ if } n=1,2,\\
16 &\text{ if } n=3,\\
11 &\text{ if } n=4,6,7,\\
0  &\text{ if } n=5,\\
3  &\text{ if } n\geq8.
\end{cases}
\]
In particular, if $d=2$ or $d=F_n+2$ for some $n\geq8$, then $A(d)=i(d)=3$.
\end{prop}
Again, the proof is left as an exercise in \walnut{}.\hfill\qed
\begin{rem}
    The observant reader may have sought some special families of values of $d$ for which $A(d)=3$ in Figure \ref{fig_Ad3fib}. One such loop in Figure \ref{fig_Ad3fib} reads the accepted value of $d$ in Zeckendorf representation as $F_n+2$ as mentioned in Proposition \ref{PROP:dfn+2fib}. Note that there exist other values of $d$ for which $A(d)=i(d)=3$, not covered by the above proposition. For example, $d=70\neq F_n+2$ for some $n$, while $A(70)=i(70)=3$. The classification of all such values can be done in \walnut{} in a similar fashion to Proposition \ref{PROP:ad3fib}.
\end{rem}
Note that this is not the full extent of our investigation of the values of $A(d)$. In the next section we address a method that applies to all values of $d$ simultaneously. Moreover, the techniques that we introduce can also be applied to determining the values of $i(d)$ (see Section \ref{SEC:altproofs}).

One can formulate similar results to Propositions \ref{PROP:dfn-1fib}, \ref{PROP:dfn+1fib}, and \ref{PROP:dfn+2fib} for $A(F_n\pm k)$ and $i(F_n\pm k)$ for any choice of $k\geq 1$. Notice that we make no mention of the case when $k=0$, because $A(F_n)$ is not eventually constant, as will be shown in the next section. Nevertheless, there appears to be a simple relationship describing the values of $i(d)$ in this case. Empirical evidence generated by \walnut{} suggests making the following Conjecture \ref{CONJ:idfn}. Unfortunately, because $A(F_n)$ is not eventually constant, we once again run into the limitation of \walnut{}'s inability to handle the multiplication of two variables.
\begin{conj}\label{CONJ:idfn}
    We have $i(F_{2n+1}) = F_{2n+3} - 2$, and $i(F_{2n})=F_{4n}-1$.
\end{conj}

\section{Dynamically proved results for MAPs in the {F}ibonacci word}\label{SEC:results-dynamics}
\subsection{Maxima of longest MAP lengths}\label{SEC:rot}


In Figure \ref{FIG:tauplot}, the values of $A(d)$ are plotted for the first few values of $d$. 
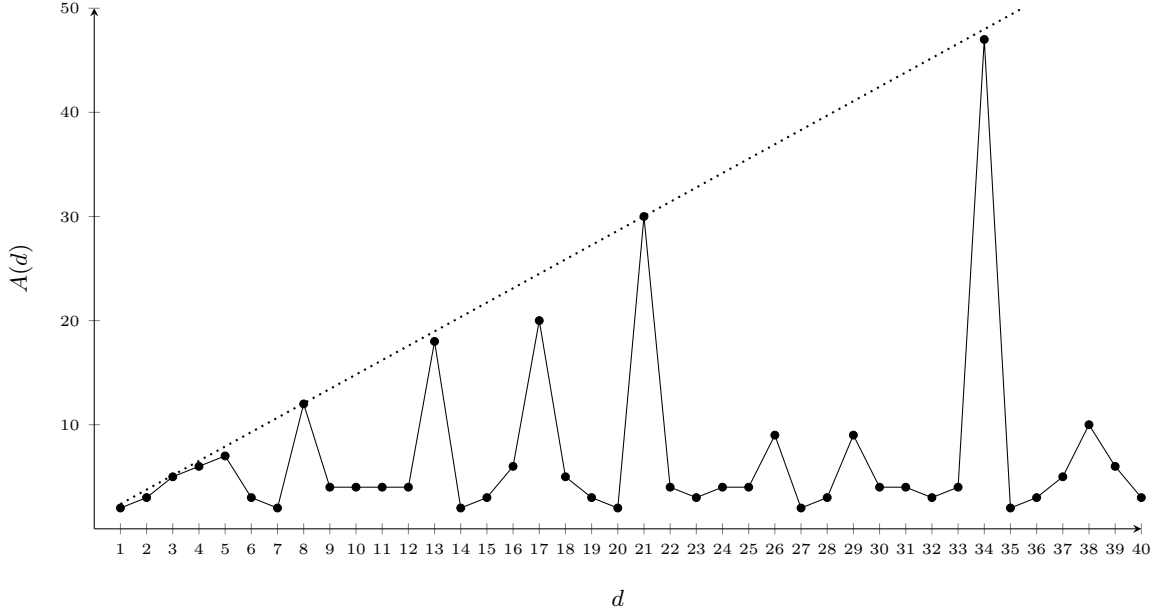
\begin{figure}[ht]
\centering
\begin{tikzpicture}[scale=0.98]
\begin{axis}[
    xmin=0, xmax=40,
    ymin=0, ymax=50,
    xtick distance=1,
    x=10,
    axis lines=center,
    every axis/.append style={font=\tiny},
    axis on top=true,
    x label style={at={(axis description cs:0.5,-0.1)},anchor=north,font=\small},
    y label style={at={(axis description cs:-0.05,0.5)},rotate=90,anchor=south,font=\small},
    xlabel=$d$,
    ylabel=$A(d)$,
    ]
\addplot[only marks, mark size=1.5pt]
    coordinates {
    (1,2)(2,3)(3,5)(4,6)(5,7)(6,3)(7,2)(8,12)(9,4)(10,4)(11,4)(12,4)(13,18)(14,2)(15,3)(16,6)(17,20)(18,5)(19,3)(20,2)(21,30)(22,4)(23,3)(24,4)(25,4)(26,9)(27,2)(28,3)(29,9)(30,4)(31,4)(32,3)(33,4)(34,47)(35,2)(36,3)(37,5)(38,10)(39,6)(40,3)
    };
\addplot[draw=black, 
    ]
    coordinates {
    (1,2)(2,3)(3,5)(4,6)(5,7)(6,3)(7,2)(8,12)(9,4)(10,4)(11,4)(12,4)(13,18)(14,2)(15,3)(16,6)(17,20)(18,5)(19,3)(20,2)(21,30)(22,4)(23,3)(24,4)(25,4)(26,9)(27,2)(28,3)(29,9)(30,4)(31,4)(32,3)(33,4)(34,47)(35,2)(36,3)(37,5)(38,10)(39,6)(40,3)
    };
     \addplot[thick, dotted, domain=1:40, samples=40] ({\x},{sqrt(5)*0.618*\x+1});
\end{axis}
\end{tikzpicture}
\caption{Plot of $A(d)$ for the first few values of $d$; the overall-peak values of $A(d)$ seem to appear at $d=F_n$ as shown with a dotted line.}
\label{FIG:tauplot}
\end{figure}

Let $k\in\N$.
We observe that local maxima of $A(d)$ appear to occur at values $d=F_n$.
Other `locally large' values that are not themselves Fibonacci numbers seem to occur at integer multiples and fractions of Fibonacci numbers, namely when $d=k F_n$ and $d=F_n/k$ for $k \geq 1$. For instance, one can check $A(d)$ values in Table \ref{tab:fibdadid} for $d=17=F_9/2,\,26=2F_7,\,29=F_{14}/13,\, 72=F_{12}/2$. Our goal in this section is to make these observations precise and to use this to completely classify the values of $A(d)$.

We begin with the case that $k=1$ and so $d=F_n$ is a Fibonacci number.

\subsubsection{Rotation sequences}
We recall some basic notation. For a real number $\alpha$, let $\{ \alpha\}$ denote its fractional part. That is, $\alpha=\lfloor \alpha \rfloor + \{ \alpha\}$, where $\lfloor\alpha\rfloor = \max\{n \in \Z \mid n\leq \alpha\}$ is the floor function. We also recall the following well-known identities describing relationships between consecutive Fibonacci numbers $F_n$ for $n\geq2$ involving the golden ratio $\tau = \frac{1}{2}(1+\sqrt{5})$.
\begin{align}
    \tau F_{n-1}+F_{n-2}&=\tau^n\label{aedoequi}\\
    F_{n+1}-\tau F_{n}& = \frac{(-1)^n}{\tau^n}\label{EQ:fibfraclim}
\end{align}

For $d=F_n$, we observe that $A(F_n)$ appears to be given by $\lceil \tau^{n-1} \rceil$.
It is unclear if this statement can be encoded in first-order Presburger arithmetic, and so we are unable to employ automatic tools like \walnut{} in attempting to demonstrate this equality.

The next result provides an equivalent formulation of the Fibonacci word as a {\em Sturmian} word, also known as a \emph{rotation sequence}.
Let $I_0 = (\tau^{-2},1)$ and $I_1 = [0,\tau^{-2})$. These intervals correspond to the reading of a $\mathtt{0}$ or $\mathtt{1}$ in the Sturmian word of slope $\tau^{-1}$.
\begin{prop}[{\cite[Prop.\ 6.1.17]{npfogg}}]\label{fibdefint}
    For all $m\geq0$, the $m$-th term in the Fibonacci word $\bff$ is given by
\[
f_m =
\begin{cases}
    \mathtt{0} & \text{if } \{ (m+1)\tau\}\in I_0\\
    \mathtt{1} & \text{if } \{ (m+1)\tau\}\in I_1.
\end{cases}
\]
\end{prop}
For example, as shown in Figure \ref{FIG:fibgenalgo}, the positions of $\{ (m+1)\tau\}$ for $m=0,1,2,3,4,5$ in $[0,1)$ give $f_0=\mathtt{0}, f_1=\mathtt{1}, f_2=\mathtt{0}, f_3=\mathtt{0}, f_4=\mathtt{1},f_5=\mathtt{0}$ respectively.
\begin{figure}[ht]
\centering
\begin{tikzpicture}[scale=12]
    \draw (0,0)-- (0.382,0);
    \draw (0.382,0)-- (1,0);
    \foreach \x in {0,1} {
        \draw (\x,0.02) -- (\x,-0.02) node[below] {\x};
        \draw (0.382,0.02) -- (0.382,-0.02);
    }
    \filldraw[black] (0.618,0) circle (0.2pt);
    \node[below] at (0.618,-0.02) {$\{ \tau\}$};
    \filldraw[black] (0.236,0) circle (0.2pt);
    \node[below] at (0.236,-0.02) {$\{ 2\tau\}$};
    \filldraw[black] (0.854,0) circle (0.2pt);
    \node[below] at (0.854,-0.02) {$\{ 3\tau\}$};
    \filldraw[black] (0.472,0) circle (0.2pt);
    \node[below] at (0.472,-0.02) {$\{ 4\tau\}$};
    \filldraw[black] (0.0902,0) circle (0.2pt);
    \node[below] at (0.0902,-0.02) {$\{ 5\tau\}$};
    \filldraw[black] (0.708,0) circle (0.2pt);
    \node[below] at (0.708,-0.02) {$\{ 6\tau\}$};
    \draw[decorate,decoration={brace,amplitude=20pt}] (0,0) -- (0.382,0) node[midway,above,yshift=20pt] {$I_1$};
    \draw[decorate,decoration={brace,amplitude=20pt}] (0.382,0) -- (1,0) node[midway,above,yshift=20pt] {$I_0$};
\end{tikzpicture}
\caption{Generating the Fibonacci word using Proposition \ref{fibdefint}.}
\label{FIG:fibgenalgo}
\end{figure}
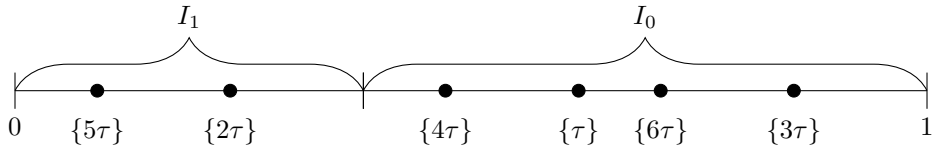

The next lemma follows easily from the formulation of $\bff$ as a rotation sequence in Proposition \ref{fibdefint}. 
\begin{lem}\label{LEM:adobs}
Fix $d\geq1$. There exists a length-$\ell$ MAP of difference $d$ on the symbol $j\in\{\mathtt{0},\mathtt{1}\}$ in $\bff$ starting at $f_{i-1}$ if and only if all of the points $\{ i\tau\}, \{ (i+d)\tau\},  \{ (i+2d)\tau\},\ldots, \{ (i+(\ell -1)d)\tau\}$ belong to the interval $I_j$.
\end{lem}

The next examples illustrate this correspondence given in Lemma \ref{LEM:adobs} between segments of orbits and MAPs.

\begin{example}\label{EX:f4}
For $d=F_4=3$, $p\geq0$ and $i=21$, the maximum number of successive points in $I_0$ are $5$, and the points $\{ i-F_4\tau\}$ and $\{ i+\ell F_4\tau\}$ are in $I_1$, as shown in Figure \ref{FIG:f4ex}. That is, $A(F_4)=5$.
\begin{figure}[ht]
\centering
\begin{tikzpicture}[scale=12]
    \draw (0,0)-- (0.382,0);
    \draw (0.382,0)-- (1,0);
    \foreach \x in {0,1} {
        \draw (\x,0.02) node[above] {\x}-- (\x,-0.02);
        \draw (0.382,0.02) -- (0.382,-0.02);
    }
    \filldraw[black] (0.125,0) circle (0.15pt);
    \node[below] at (0.125,-0.02) {\small $\{ 18\tau\}$};
    \filldraw[black] (0.979,0) circle (0.15pt);
    \node[below] at (0.979,-0.02) {\small $\{ 21\tau\}$};
    \filldraw[black] (0.833,0) circle (0.15pt);
    \node[below] at (0.833,-0.02) {\small $\{ 24\tau\}$};
    \filldraw[black] (0.687,0) circle (0.15pt);
    \node[below] at (0.687,-0.02) {\small $\{ 27\tau\}$};
    \filldraw[black] (0.541,0) circle (0.15pt);
    \node[below] at (0.541,-0.02) {\small $\{ 30\tau\}$};
    \filldraw[black] (0.395,0) circle (0.15pt);
    \node[below] at (0.395,-0.02) {\small $\{ 33\tau\}$};
    \filldraw[black] (0.249,0) circle (0.15pt);
    \node[below] at (0.249,-0.02) {\small $\{ 36\tau\}$};
    \draw[decorate,decoration={brace,amplitude=20pt}](0,0) -- (0.382,0) node[midway,above,yshift=20pt] {$\tau^{-2}$};
    \draw[decorate,decoration={brace,amplitude=20pt}] (0.382,0) -- (1,0) node[midway,above,yshift=20pt] {$\tau^{-1}$};
\end{tikzpicture}
\caption{A segment of the orbit $(\{(18+3n)\tau\})_{n\geq 0}$ illustrating that $A(3)=5$.}
\label{FIG:f4ex}
\end{figure}
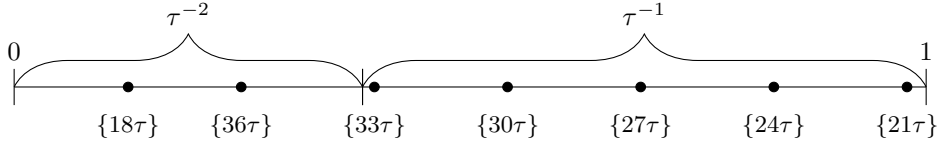
\end{example}
\begin{example}\label{EX:f5}
    For $d=F_5=5$ and $i=12$, we get $A(F_5)=7$ as shown in Figure \ref{FIG:f5ex}.
\begin{figure}[ht]
\centering
\begin{tikzpicture}[scale=12]
    \draw (0,0)-- (0.382,0);
    \draw (0.382,0)-- (1,0);
    \foreach \x in {0,1} {
        \draw (\x,0.02) node[above] {\x}-- (\x,-0.02);
        \draw (0.382,0.02) -- (0.382,-0.02);
    }
    \filldraw[black] (0.326,0) circle (0.15pt);
    \node[below] at (0.326,-0.02) {\small $\{ 7\tau\}$};
    \filldraw[black] (0.416,0) circle (0.15pt);
    \node[below] at (0.416,-0.02) {\small $\{ 12\tau\}$};
    \filldraw[black] (0.507,0) circle (0.15pt);
    \node[below] at (0.507,-0.02) {\small $\{ 17\tau\}$};
    \filldraw[black] (0.597,0) circle (0.15pt);
    \node[below] at (0.597,-0.02) {\small $\{ 22\tau\}$};
    \filldraw[black] (0.687,0) circle (0.15pt);
    \node[below] at (0.687,-0.02) {\small $\{ 27\tau\}$};
    \filldraw[black] (0.777,0) circle (0.15pt);
    \node[below] at (0.777,-0.02) {\small $\{ 32\tau\}$};
    \filldraw[black] (0.867,0) circle (0.15pt);
    \node[below] at (0.867,-0.02) {\small $\{ 37\tau\}$};
    \filldraw[black] (0.957,0) circle (0.15pt);
    \node[below] at (0.957,-0.02) {\small $\{ 42\tau\}$};
    \filldraw[black] (0.048,0) circle (0.15pt);
    \node[below] at (0.048,-0.02) {\small $\{ 47\tau\}$};
    \draw[decorate,decoration={brace,amplitude=20pt}] (0,0) -- (0.382,0) node[midway,above,yshift=20pt] {$\tau^{-2}$};
    \draw[decorate,decoration={brace,amplitude=20pt}] (0.382,0) -- (1,0) node[midway,above,yshift=20pt] {$\tau^{-1}$};
\end{tikzpicture}
\caption{A segment of the orbit $(\{(12+5n)\tau\})_{n\geq 0}$ illustrating that $A(5)=7$.}
\label{FIG:f5ex}
\end{figure}
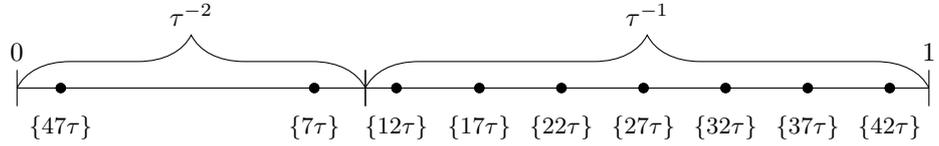
\end{example}

By {\em Weyl's theorem}, as $d\tau$ is irrational for all $d \geq 1$, the sequence of points $(\{ (i+md)\tau\})_{m \geq 0}$ are dense and uniformly distributed in $[0,1)$. 
Then an immediate consequence of Lemma \ref{LEM:adobs} and Weyl's theorem is that MAPs in $\bff$ are all finite in length, as any orbit of the form $(\{(i+nd)\tau\})_{n \geq 0}$ must transition between the intervals $I_0$ and $I_1$ infinitely often. This provides a simple proof in the special case of $\bff$ (and indeed any Sturmian word) of a more general result of Durand and Goyeneche \cite[Proposition 8]{dg2019}.
\begin{prop}\label{PROP:fibadfin}
    The Fibonacci word admits no infinite MAP.\hfill\qed
\end{prop}
Therefore, $A(d)$ always takes finite values.

In order to determine exact values for $A(d)$, it is useful to analyse the pairwise distances between consecutive points of the form $\{(i+nd)\tau\}$ and $\{ (i+(n+1)d)\tau\}$ for some $n\in\Z$ in an orbit.
Let $g(d)$ denote the {\em step distance} defined by 
\begin{equation}\label{DEF:stepdistance}
g(d)=\min\left(\{ d\tau\},1-\{ d\tau\}\right).
\end{equation}
Note that $g(d)=\|d\tau\|$, where $\|x\|=\min|x-\Z|$, which is a common notation from Diophantine approximation denoting the distance to the nearest integer. One should consider two cases that are illustrated in Figure \ref{FIG:stepdisEQ}; either $\{d\tau\}<0.5$, in which case we are \emph{stepping to the right} by a distance of $g(d)$, or else $\{d\tau\}>0.5$, in which case we are \emph{stepping to the left} by a distance of $g(d)$.

\begin{figure}[ht]
\centering
\begin{tikzpicture}[scale=12]
    \draw (0,0)-- (0.382,0);
    \draw (0.382,0)-- (1,0);
    \foreach \x in {0,1} {
        \draw (\x,0.02) -- (\x,-0.02) node[below] {\x};
        \draw (0.382,0.02) -- (0.382,-0.02) node[below] {$\tau^{-2}$};
    }
    
    \draw[-{Stealth[width=2mm]}] (0.39, 0) .. controls (0.42, 0.03) and (0.45, 0.03) .. (0.48, 0);
    \draw[-{Stealth[width=2mm]}] (0.48, 0) .. controls (0.51, 0.03) and (0.54, 0.03) .. (0.57, 0)       
     node[midway, above, yshift=3pt] {$\{ d\tau\}<0.5$};
    \draw[-{Stealth[width=2mm]},dashed] (0.57, 0) .. controls (0.60, 0.03) and (0.63, 0.03) .. (0.66, 0);
    \draw[|<->|] (0.48, -0.015) -- (0.57, -0.015) node[midway, below] {$g(d)$};

    \draw[-{Stealth[width=2mm]}] (0.99, 0) .. controls (0.96, 0.03) and (0.93, 0.03) .. (0.90, 0);
    \draw[-{Stealth[width=2mm]}] (0.90, 0) .. controls (0.87, 0.03) and (0.84, 0.03) .. (0.81, 0)       
     node[midway, above, yshift=3pt] {$\{ d\tau\}>0.5$};
    \draw[-{Stealth[width=2mm]},dashed] (0.81, 0) .. controls (0.78, 0.03) and (0.75, 0.03) .. (0.72, 0);
    \draw[|<->|] (0.81, -0.015) -- (0.90, -0.015) node[midway, below] {$g(d)$};
\end{tikzpicture}
\caption{Illustrating when steps are taken to the right or the left.}
\label{FIG:stepdisEQ}
\end{figure}
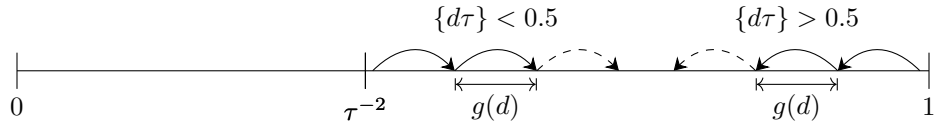

Suppose that each of the points $\{ i\tau\}, \{ (i+d)\tau\}, \{ (i+2d)\tau\}, \ldots, \{ (i+(\ell-1) d)\tau\}$ lies in the interval $I_0$, corresponding to the existence of a MAP of length $\ell$ with difference $d$ by Lemma \ref{LEM:adobs}.
If $g(d)$ is sufficiently small (such as in Figure \ref{FIG:stepdisEQ}), then consecutive pairs in this orbit segment are a distance $g(d)$ apart. So, these pairs of points are the end points of consecutive touching intervals that are all of length $g(d)$. Hence, the sum of these lengths, $\ell g(d)$, must be shorter than the length of $I_0$.
Note that $|I_1|<|I_0|$ and so to maximise the length of a MAP, we may always consider points that fall into $I_0$ (compare Lemma \ref{LEM:idnever1}). As the orbit $\{ i\tau \}$ is dense in $[0,1)$, if $\ell g(d) < |I_0|$, then there exists a point $\{ i \tau\} \in I_0$ as close as we like to an end point of the interval $I_0$ (which end point depends on if we are jumping to the left or the right) so that we can ensure that all of the points $\{ i\tau\}, \{ (i+d)\tau\}, \{ (i+2d)\tau\}, \ldots, \{ (i+(\ell-1) d)\tau\}$ also fall into the interval $I_0$. This means that $\ell g(d) <|I_0|$ implies the existence of a MAP of length $\ell$ with distance $d$.

As soon as $\{ (i+(\ell-1) d)\tau\}$ falls strictly within a distance $g(d)$ of the other end of $I_0$ from $\{i\tau\}$, the next point $\{ (i+\ell d)\tau\}$ must be in $I_1$. This situation is illustrated in Figures \ref{FIG:f4ex} and \ref{FIG:f5ex} (notice the points $\{ 33\tau\}$ and $\{ 36\tau\}$, respectively $\{ 42\tau\}$ and $\{ 47\tau\}$). Then, any orbit segment of greater length must necessarily enter the interval $I_1$ because we have chosen $g(d)$ small enough that one cannot `jump over' the interval $I_1$. This means that there is no MAP of length $\ell+1$ with distance $d$. In other words, $A(d) = \ell$.

Of course, this argument only holds if $g(d)$ is small enough, and in particular if $g(d)<\tau^{-2}$, as in this range, consecutive points in an orbit cannot jump over $I_1$ in order to land back into $I_0$. It is therefore necessary to consider two distinct cases; we either have $0\leq g(d)< \tau^{-2}$ or $\tau^{-2}\leq g(d)\leq 0.5$. The above argument holds in the first case, so we have that $A(d)$ satisfies
\[
(A(d)-1) g(d) < |I_0|\text{ and }A(d)g(d) \geq |I_0|.
\]
As $|I_0| = \tau^{-1}$, this gives the following classification of $A(d)$ in this case.
\begin{theorem}\label{THM:adformula}
    For $g(d)< \tau^{-2}$, \[A(d) = \left\lfloor\frac{|I_0|}{g(d)}\right\rfloor + 1 = \left\lceil\frac{|I_0|}{g(d)}\right\rceil = \left\lceil\frac{\tau^{-1}}{g(d)}\right\rceil.\]\hfill\qed
\end{theorem}
Note that $|I_0|/g(d)$ is never an integer. Hence, the second equality in the above holds.

It is sometimes useful to consider the two ranges that $g(d)$ falls into instead as ranges that $\{d\tau\}$ falls into. Comparing the definitions of $g(d)$ and $\{d\tau\}$ then gives the following.
\begin{lem}\label{LEM:fractarith}
We have
\begin{align*}
g(d) \in [0, \tau^{-2}) &\quad \Longleftrightarrow \quad  \{ d\tau\} \in[0,\tau^{-2})\cup (\tau^{-1},1],\\
g(d) \in [\tau^{-2},0.5] &\quad \Longleftrightarrow \quad  \{ d\tau \} \in [\tau^{-2},\tau^{-1}].
\end{align*}
\end{lem}
\begin{rem}\label{REM:I0gdnodiv}
It is easy to show that for all $d\geq 1$ the only value of $d$ for which $\{d\tau\}$ is a boundary point of any of the above intervals is when $d=1$, since $\{\tau\} = \tau^{-1}$. So, it is not so important if the intervals in the above ranges are open or closed. Similarly, one can show that $\{d\tau\}$ never divides into $\tau^{-1}$ a non-trivial integer number of times, which is useful to know when discussing $A(F_n/k)$ later.
\end{rem}
Before proving the analogue of Theorem \ref{THM:adformula} in the case $g(d) \in [\tau^{-2},0.5]$, we analyse an important family of differences that fall into the $g(d) \leq \tau^{-2}$ case; namely, we consider the case that $d$ is a Fibonacci number.
\begin{cor}\label{COR:ceiltauad}
    For all $n\geq2$, $g(F_n)=\tau^{-n}$ and $A(F_n)=\left\lceil\tau^{n-1}\right\rceil$.
\end{cor}
\begin{proof}
By Equation \ref{EQ:fibfraclim}, we have
\[
\{ F_n \tau\}=\left\{ F_{n+1}-\frac{(-1)^{n}}{\tau^{n}}\right\}.
\]
In particular, if $n$ is odd, then $ \{ F_{n}\tau\}=\tau^{-n}\in[0,\tau^{-2}]$, and if $n$ is even, then $\{ F_{n}\tau \}= \{ 1-\tau^{-n}\}=1-\tau^{-n}\in[\tau^{-1},1]$. So, by Lemma \ref{LEM:fractarith} and the definition of $g(d)$, we get 
\[
g(F_n)=\tau^{-n}.
\]
Finally, we use Theorem \ref{THM:adformula} to conclude that
\[
A(F_n)=\left\lceil\frac{\tau^{-1}}{g(F_n)}\right\rceil=\left\lceil\frac{\tau^{-1}}{\tau^{-n}}\right\rceil=\left\lceil\tau^{n-1}\right\rceil.
\]
\end{proof}
\begin{rem}
Notice that it is not always true that $A(d)=\lceil\tau^{n-1}\rceil$ implies $d=F_n$. For example, $A(6)=3$, where $3=\lceil\tau^{n-1}\rceil$ for $n=3$, but $6$ is not a Fibonacci number.
\end{rem}
\begin{rem}
    Using the theory of {\em model sets}, Aedo et al.\ provided an independent proof of Corollary \ref{COR:ceiltauad} in their unpublished work \cite{aedonote2022}. They gave the formulation $A(F_n)=F_{n-2}+\lceil\tau F_{n-1}\rceil$, which can be easily shown to be equivalent to ours by Equation \ref{aedoequi}.
\end{rem}

From Corollary \ref{COR:ceiltauad} and a simple application of Binet's formula, we obtain the following asymptotic result describing the slope shown in Figure \ref{FIG:tauplot}, which tracks the linear growth rate of $A(d)$ when $d$ is a Fibonacci number.
\begin{cor}
    \[\lim_{n\to \infty}\frac{A(F_n)}{F_n} = \lim_{n\to \infty}\frac{\left\lceil\tau^{n-1}\right\rceil}{F_n} = \sqrt{5}\tau^{-1}\approx 1.3819\ldots.\]
\end{cor}

\subsection{Multiples of {F}ibonacci numbers}
The following lemma is a natural extension of what the step distance looks like for an integer multiple of $d$, and hence will prove useful when determining $A(nd)$.
\begin{lem}\label{LEM:ngofd}
We have \[g(nd)=\min(\{ ng(d)\},1-\{ ng(d)\}).\]
\end{lem}
\begin{proof}
First, it is well known that for $n\in\Z$ and $x\in\R$, 
\begin{align*}
\{ n\{ x\}\}&=\{ nx \}\qquad\text{and}\\
\{-x\}		&=1-\{x\}.
\end{align*}

Now, by definition of step distance from Equation \ref{DEF:stepdistance},
\begin{align*}
g(nd)&=\min(\{ nd\tau\}, 1-\{ nd\tau\})\\
     &=\min(\{ n\{ d\tau\}\}, 1-\{ n\{ d\tau\}\}).
\end{align*}
If $g(d)=\{ d\tau\}$, we have
\[g(nd)=\min(\{ ng(d)\}, 1-\{ ng(d)\});\]
and if $g(d)=1-\{ d\tau\}$, we have
\begin{align*}
g(nd)&=\min(\{ n(1-g(d))\}, 1-\{ n(1-g(d))\})\\
     &=\min(\{ -ng(d)\}, 1-\{ -ng(d)\})\\
     &=\min(1-\{ ng(d)\}, 1-(1-\{ ng(d)\}))\\
     &=\min(1-\{ ng(d)\}, \{ ng(d)\}).
\end{align*}
\end{proof}

The following proposition uses Theorem \ref{THM:adformula} and Corollary \ref{COR:ceiltauad} to generalise the formula for $A(F_n)$ to $A(kF_n)$ given a small enough multiplier $k$ and large enough index $n$.
\begin{prop}
Let $k\in\N$ and $N$ be such that $k\leq \tau^{N-2}$. For all $n\geq N$,
    \[A(kF_n)=\left\lceil\frac{A(F_n)}{k}\right\rceil=\left\lceil \frac{\tau^{n-1}}{k}\right\rceil.\]
\end{prop}
\begin{proof}
As $k\leq \tau^{N-2}$, then $g(F_n)\leq \tau^{-2}/k$ for all $n \geq N$. We therefore also have $kg(F_n) = g(kF_n)$, as $kg(F_n) \leq 0.5$. So, $g(kF_n) \leq \tau^{-2}$ and we can apply Theorem \ref{THM:adformula}.

This gives
\[
A(kF_n) = \left\lceil\frac{\tau^{-1}}{g(kF_n)}\right\rceil = \left\lceil\frac{\tau^{-1}}{kg(F_n)}\right\rceil = \left\lceil\frac{\tau^{n-1}}{k}\right\rceil = \left\lceil\frac{A(F_n)}{k}\right\rceil.
\]
\end{proof}

We now classify $A(d)$ for the remaining case that $g(d) \in [\tau^{-2},0.5]$.
\begin{theorem}\label{THM:adformula2}
If $d \geq 2$ and $\tau^{-2} \leq g(d) \leq 0.5$, then
\[A(d)=2\left\lceil \frac{\tau^{-1}-g(d)}{g(2d)}\right\rceil.\]
\end{theorem}
\begin{proof}
Without loss of generality, we may assume that $\{d\tau\}<0.5$ so that we are jumping to the right. The case $\{d\tau\}>0.5$ is completely analogous.
Choose $i$ such that $1-g(d)-g(2d)<\{i\tau\} < 1-g(d)$.
Then we have $\{ i\tau\}+g(d)<1$ and $\{ i\tau\}-kg(2d)\geq \tau^{-2}$ for some $k$ and by construction, no other starting interval satisfies these inequalities for a larger value of $k$. That is, $k$ is maximal with this property. The first inequality means that the next point $\{ i\tau\}+g(d)=\{(i+d)\tau\}$ lands just before the point $1$, and so remains in $I_0$. The next point $\{ (i+2d)\tau\}$ avoids falling in $I_1$ because $g(d) >\tau^{-2}$ and $\{(i+d)\tau\}$ is within some small enough $\epsilon$ of the point $1$. See Figure \ref{FIG:stepdishopoverI1}.

\begin{figure}[ht]
\centering
\begin{tikzpicture}[scale=12]
    \draw (0,0)-- (0.382,0);
    \draw (0.382,0)-- (1,0);
    \foreach \x in {0,1} {
        \draw (\x,0.02) -- (\x,-0.02) node[below] {};
        \draw (0.382,0.02) -- (0.382,-0.02) node[below] {$\tau^{-2}$};
        \draw (0.52,0.01) -- (0.52,-0.01);
        \draw[dotted,thick] (0.52,0.01) -- (0.52,0.06) node[above] {${\scriptstyle 1-g(d)}$};
        \draw[dotted,thick] (0.50,0) -- (0.50,-0.03) node[below] {$\{i\tau\}$};
        \draw[dotted,thick] (0.98,0) -- (0.98,-0.03) node[below] {$\{(i+d)\tau\}$};
    }
    
    \draw[-{Stealth[width=2.5mm]}]  (0.50, 0) .. controls (0.54, 0.07) and (0.94, 0.07) .. (0.98, 0);
    \filldraw[black] (0.50,0) circle (0.1pt);
    \filldraw[black] (0.98,0) circle (0.1pt);
    
    \filldraw[black] (0.455,0) circle (0.1pt);
    \filldraw[black] (0.41,0) circle (0.1pt);
    \node at (0.365,0) {$\times $};

    \draw[|<->|] (0.455, 0.015) -- (0.50, 0.015) node[midway, above] {${\scriptstyle g(2d)}$};    

    \filldraw[black] (0.935,0) circle (0.1pt);
    \filldraw[black] (0.89,0) circle (0.1pt);

    \draw[-]  (0.98, 0) .. controls (0.983, 0.015) .. (1, 0.025);
    \draw[-,dotted,thick]  (1, 0.025) .. controls (1.01, 0.030) .. (1.05, 0.038);
    \draw[-,dotted,thick]  (-0.02, 0) .. controls (-0.017, 0.015) .. (0, 0.025);
    \draw[-{Stealth[width=2.5mm]}]  (0, 0.025) .. controls (0.08, 0.075) and (0.43, 0.055) .. (0.455,0);
\end{tikzpicture}
\caption{Two clusters of points in the orbit segment form when $g(d)>\tau^{-2}$ due to the ability for consecutive points to `jump over' the shorter interval $I_1$.}
\label{FIG:stepdishopoverI1}
\end{figure}

Now, by Lemma \ref{LEM:ngofd}, we know that every $(n+2)$-th point falls $g(2d)$ to the left of the $n$-th point. This continues until a point $\{(i+2kd)\tau\} = \{i\tau\}-2kg(2d)$ falls within a distance $g(2d)$ to the right of $\tau^{-2}$. Then, the next point $\{ (i+(2k+1)d)\tau\}$ is the last consecutive point that appears in $I_0$ because we necessarily have $\{ (i+(2k+2)d)\tau\}<\tau^{-2}$ by the maximality of $k$. 

So, the points in the orbit segment fall into one of two equal-sized clusters of points that appear towards either end of the interval $I_0$ (see, e.g., Figure \ref{FIG:d148ex}), with each cluster containing exactly $k+1$ points. It follows that $A(d)\geq 2(k+1)$. Since we have
\[
k\leq \frac{\{ i\tau\} -\tau^{-2}}{g(2d)}\leq \frac{1-g(d) -\tau^{-2}}{g(2d)}=\frac{\tau^{-1}-g(d)}{g(2d)},
\]
then for $k$ to be maximal, we have
\[
k = \left\lfloor \frac{\tau^{-1}-g(d)}{g(2d)} \right\rfloor,
\]
which gives
\[
A(d) \geq 2\left(\left\lfloor \frac{\tau^{-1}-g(d)}{g(2d)} \right\rfloor +1\right) = 2\left\lceil \frac{\tau^{-1}-g(d)}{g(2d)} \right\rceil
\]
because $g(2d)$ never divides $\tau^{-1}-g(d)$ an integer number of times for $d \geq 2$. This can be seen by observing that $\tau^{-1}-g(d)=\tau^{-1}-0.5+\frac{1}{2}g(2d)$ together with the fact that $g(2d)\in \mathbb{Z}[\tau]$.

In order to see that in fact this is the maximal value of $A(d)$, we notice that any other starting point gives rise to an orbit segment that either satisfies the same conditions as $\{i\tau\}$ (so gives rise to a MAP of the same length), or else must enter $I_1$ earlier because the orbit segment will be a translate of the orbit segment starting at position $\{i\tau\}$ and $k$ was chosen maximally.
\end{proof}

So, Theorems \ref{THM:adformula} and \ref{THM:adformula2} together allow us to determine $A(d)$ for any $d\geq 1$.

Note that if $g(d) \in [\tau^{-2},0.5]$, then $g(2d)=1-2g(d)$, so $A(d)$ can also be written as either
\[
A(d) = 2\left\lceil \frac{\tau^{-1}-g(d)}{1-2g(d)}\right\rceil\quad \text{or}\quad A(d)=2\left\lceil \frac{\tau^{-1}-0.5+\frac{1}{2}g(2d)}{g(2d)}\right\rceil,
\]
so one only needs to compute either $g(d)$ or $g(2d)$, not both.
\begin{example}
Given $d=148$, we get $A(148)=6$ and Figure \ref{FIG:d148ex} shows the distribution of points starting at $i=17$. The step distance between the adjacent points in each cluster is $g(296)$.

\begin{figure}[ht]
\centering
\begin{tikzpicture}[scale=14.8]
    \draw (0,0)-- (0.382,0);
    \draw (0.382,0)-- (1,0);
    \foreach \x in {0,1} {
        \draw (\x,0.02) node[above] {\x}-- (\x,-0.02);
        \draw (0.382,0.02) -- (0.382,-0.02);
    }
    \filldraw[black] (0.507,0) circle (0.1pt);
    \node[below] at (0.507,-0.02) {\small $0$};
    \filldraw[black] (0.976,0) circle (0.1pt);
    \node[below] at (0.976,-0.02) {\small $1$};
    \filldraw[black] (0.445,0) circle (0.1pt);
    \node[below] at (0.445,-0.02) {\small $2$};
    \filldraw[black] (0.914,0) circle (0.1pt);
    \node[below] at (0.914,-0.02) {\small $3$};
    \filldraw[black] (0.387,0) circle (0.1pt);
    \node[below] at (0.387,-0.02) {\small $4$};
    \filldraw[black] (0.852,0) circle (0.1pt);
    \node[below] at (0.852,-0.02) {\small $5$};
    \filldraw[black] (0.321,0) circle (0.1pt);
    \node[below] at (0.321,-0.02) {\small $6$};
    \filldraw[black] (0.038,0) circle (0.1pt);
    \node[below] at (0.038,-0.02) {\small $-1$};
    \draw[decorate,decoration={brace,amplitude=20pt}] (0,0) -- (0.382,0) node[midway,above,yshift=20pt] {$\tau^{-2}$};
    \draw[decorate,decoration={brace,amplitude=20pt}] (0.382,0) -- (1,0) node[midway,above,yshift=20pt] {$\tau^{-1}$};
\end{tikzpicture}
\caption{Counting the repetition of $\mathtt{0}$ to find $A(148)$; the points $\{ (i+nd)\tau\}$ are denoted only by $n$ for convenience.}
\label{FIG:d148ex}
\end{figure}
\end{example}

It is apparent from Figure \ref{FIG:tauplot} that the values $A(F_n)$ achieve new maxima of $A(d)$ over all $1\leq d\leq F_n$.
To prove this, we employ a version of the three-gap theorem in the specific case of the rotation being given by $\tau$. A suitable reference for this special case appears in work of Ravenstein \cite[Section 3]{ravenstein1989}, which we provide here. 
\begin{theorem}[\cite{ravenstein1989}]\label{THM:3gap}
The set of points $P_d = \{\{ i\tau \}\mid1\leq i\leq d\}$ partition the circle $[0,1]/\{0,1\}$ into intervals taking at most $3$ distinct lengths. If $F_{n-1}\leq d<F_n$, then the smallest length is $|F_{n-1}\tau-F_n|$ and $\{ F_{n-1}\tau\}$ is the closest point in $P_d$ to the boundary of the interval $[0,1]$.
\end{theorem}
\begin{cor}\label{COR:newmaximafn}
    For every $n$, a new maximum of $A(d)$ over all $1 \leq d \leq F_n$ is achieved at $d=F_n$.
\end{cor}
\begin{proof}
We are required to show that if $d < F_n$, then $A(d) < A(F_n)$.
We first consider the case that $g(d)<\tau^{-2}$.
By Theorem \ref{THM:3gap}, we know that $g(F_n) < g(d)$ as $g(d)$ is exactly the distance of $\{d\tau\}$ from the boundary of $[0,1]$. Then, by Theorem \ref{THM:adformula} we find that $A(d) \leq A(F_n)$ as required.

In the case that $g(d) \in [\tau^{-2},0.5]$, we may assume that $\{d\tau\}<0.5$ as the other case is analogous. Then, $g(d) = 0.5-\frac{1}{2}g(2d)$ and so Theorem \ref{THM:adformula2} gives us
\[
A(d) = 2\left\lceil\frac{\tau^{-1}-0.5+\frac{1}{2}g(2d)}{g(2d)}\right\rceil = 2\left\lceil\frac{\tau^{-1}-0.5}{g(2d)}+\frac{1}{2}\right\rceil.
\]
As $d<F_n$, then $2d<F_{n+2}$ meaning that $g(2d) \geq g(F_{n+1}) = \tau^{-(n+1)}$ by Theorem \ref{THM:3gap}, which gives
\begin{align*}
A(d) \leq 2\left\lceil(\tau^{-1}-0.5)\tau^{n+1}+\frac{1}{2}\right\rceil &\leq (2\tau^{-1}-1)\tau^{n+1}+2\\
	& = \tau^{-3}\tau^{n+1}+2\\
	& = \tau^{n-2}+2.
\end{align*}
For all $n \geq 5$, $\tau^{n-2}+2<\tau^{n-1} < A(F_n)$.
For $n < 5$ (that is, $d < 8$), one checks by hand, for instance in Table \ref{tab:fibdadid}.
\end{proof}

\subsection{The {P}isano period, rank of apparition and factors of {F}ibonacci numbers}
For $k\geq 1$, let $\pi(k)$ denote the \emph{Pisano period} of $k$, the minimal period of the sequence $P_k = (F_i \pmod k)_{i \geq 0}$. It is known that for all $k \geq 1$, $P_k$ is periodic, and so $\pi(k)$ is well-defined and finite \cite{vajda1989}. A related quantity, the \emph{rank of apparition} of $k$, or \emph{rank} for short, is defined to be the smallest $n \geq 1$ such that $F_n\equiv 0 \pmod{k}$, that is, the smallest index $n\geq 1$ for which $P_n=0$. We let $\alpha(k)$ denote the rank of $k$. Note that $\alpha(k)$ is well-defined and finite. It is also known that $F_n \equiv 0 \pmod{k}$ if and only if $\alpha(k)|n$.

It is therefore of interest, for some fixed $k\geq 2$, to study $A(d)$ for the family of values $d=F_{n\alpha(k)}/k$. Of particular note is the fact that every natural number is of this form. That is, every natural number divides a Fibonacci number. In this section, we show that for a fixed $k\geq 3$, $A(F_n/k)$ takes only finitely many values over all $n$ for which $k|F_n$.

\subsubsection{The case of $k = 2$}
We first address the case of $k=2$, i.e. $d=F_n/2$, as this case differs from $k \geq 3$. This case also happens to give the newest peak values of $A(d)$ when $\tau^{-2}<g(d)<0.5$. Keep in mind that $\alpha(2)=3$, and so a Fibonacci number is even if and only if the index $n$ is a multiple of $3$.

Note that for any real number $x$, $\{x/2\}$ must either be $\{x\}/2$ or $\{x\}/2+0.5$. Hence, $\{F_n\tau/2\}$ must be $\{F_n \tau\}/2+0.5$, as Theorem \ref{THM:3gap} precludes $\{F_n\tau/2\}$ from being smaller than $\{F_n\tau\}$. Therefore, for $n \geq 3$,
\[
g(F_n/2) = 1-\{F_n\tau/2\} = 1-(\{F_n\tau\}/2+0.5)=0.5-g(F_n)/2.
\]
It follows that $g(F_n/2) \to 0.5$ as $n \to \infty$.

From the formulation in Theorem \ref{THM:adformula2}, we get the following.
\begin{cor}\label{COR:A(fn/2)}
    For $d=F_n/2 \geq 2$ an integer, 
    \[A(F_n/2)=2\left\lceil\frac{1}{2}\tau^{n-3}+1/2\right\rceil = \lceil \tau^{n-3}\rceil + 1 +\left(\frac{n}{3}\mod 2\right).\]
\end{cor}
\begin{proof}
The first value of $n$ for which $F_n/2 \geq 2$ is an integer is $n=6$, giving $g(F_6/2) \in [\tau^{-2},0.5]$ and so Theorem \ref{THM:adformula2} applies. For all greater $n$, as $g(F_n)$ is monotonically decreasing, then $g(F_n/2) = 0.5 - g(F_n)/2 \in [\tau^{-2},0.5]$ also. So, we get
\[A(F_n/2)=2\left\lceil\frac{\tau^{-1}-g(F_n/2)}{g(F_n)}\right\rceil =2\left\lceil\frac{\tau^{-1}-0.5+g(F_n)/2}{g(F_n)}\right\rceil =2\left\lceil\frac{1}{2}\tau^{n-3}+1/2\right\rceil.\]
The final equality follows from the fact that $\{\tau^{2n}\}>0.5$ and $\{\tau^{2n+1}\}<0.5$ for all $n \geq 2$.
\end{proof}
Note that this formula simplifies to $\lceil \tau^{n-3}\rceil +1 $ if $n \equiv 0\pmod 6$ and simplifies to $\lceil \tau^{n-3}\rceil +2 $ if $n \equiv 3\pmod 6$.

By applying Binet's formula, the above result gives an asymptotic for the slope of the line passing through the peaks of $A(d)$ for $d = F_n/2$. We get
\begin{cor}\label{COR:Afn/2bound}
\[\lim_{n\to\infty}\frac{A(F_n/2)}{F_n/2}= \lim_{n\to\infty} 2\left\lceil\frac{1}{2}\tau^{n-3}+1/2\right\rceil\frac{2}{F_n} = 2\sqrt{5}\tau^{-3}\approx 1.0557\ldots .\qed\]
\end{cor}
Using an analysis similar to that in the proof of Corollary \ref{COR:newmaximafn}, one can show that for all $d<F_n/2$ such that $g(d) \in [\tau^{-2},0.5]$, $A(d)$ is bounded above by $A(F_n/2)$.
\subsubsection{The case of $k \geq 3$}
Notice that in the case of $k=2$, we applied Theorem \ref{THM:adformula2} because $g(d) \in [\tau^{-2},0.5]$. This essentially amounts to the observation that $\{F_n\tau/2\} \to 0.5$ as $n \to \infty$. In the case that $k \geq 3$, it would therefore be useful if we could show that $g(F_n/k)\to 1/k$. Unfortunately this is not always the case. For example, $g(F_n/5)$ contains two convergent subsequences; one that converges to $1/5$ and one that converges to $2/5$. In general, we can only guarantee that the sequence $g(F_n/k)$ admits a finite set of cluster points that is a subset of $\{i/k \mid 0 \leq i \leq \lfloor k/2\rfloor\}$. This follows from the fact that $\{g(F_n/k)k\}\to 0$, and so $g(F_n/k)$ must have fractional part close to $i/k$ for some $0 \leq i \leq \lfloor k/2\rfloor$.

Note that $0$ is still precluded from being a cluster point by Theorem \ref{THM:3gap} --- here, one needs to be careful to show that one cannot have $F_n/k$ being a Fibonacci number itself infinitely often (as otherwise Theorem \ref{THM:3gap} does not readily apply in the same way). One way to see this is to use the fact that
the ratio $F_m/F_n$ is approximately $\tau^{m-n}$. Indeed Suppose that there are infinitely many $n_i$ for which $F_{n_i}/k=F_{m_i}$ for some $m$. Then $k=F_{n_i}/F_{m_i}$ for all $i$, but by Binet's formula, $F_{n_i}/F_{m_i} = \tau^{n_i-m_i}+O(\tau^{-n_i})$. Notice that eventually $n_i-m_i$ must be constant because $\tau^{-n_i}$ converges to 0. Hence, we have $k=\lim_{i \to \infty} F_{n_i}/F_{m_i} = \tau^{N_k}$ for some constant $N_k > 0$. However, this is a contradiction as $k$ is an integer, while $\tau^{N_k}$ is irrational. We could have also used the fact that $F_{m_i}|F_{n_i}$ if and only if $m_i|n_i$, which can only hold for finitely many pairs $(n_i,m_i)$, as $n_i-m_i$ is eventually constant and equal to $N_k$. Hence $m_i|m_i+N_k$, meaning $m_i|N_k$. However, $N_k$ admits only finitely many factors. We conclude that $0$ cannot be a cluster point of $g(F_n/k)$ and so the set of possible cluster points is a subset of $\{i/k \mid 1 \leq i \leq \lfloor k/2\rfloor\}$. We therefore have the following result.

\begin{prop}
Let $k \geq 3$ be fixed. And let $\alpha(k)$ be the rank of $k$. The sequence $(g(F_{n\alpha(k)}/k))_{n \geq 1}$ admits at most $\lfloor k/2\rfloor$ cluster points, none of which is $0$.\qed
\end{prop}

Similar reasoning shows that $g(2F_n/k)$ also admits finitely many cluster points, none of which is $0$. In particular, Theorems \ref{THM:adformula} and \ref{THM:adformula2} give $A(d)$ as functions that only take large values in the case that either $g(d)$ or $g(2d)$ are small. However, in both cases, $g(F_n/k)$ and $g(2F_n/k)$ never have $0$ as a cluster point, and so $A(F_{n\alpha(k)}/k)$ is uniformly bounded for a fixed $k$.

\begin{prop}
Let $k \geq 3$ be fixed and let $\alpha(k)$ be the rank of $k$. $\{A(F_{n\alpha(k)}/k)\mid n \geq 1\}$ takes finitely many values.\qed
\end{prop}

Thankfully, in the case that $k=3$, we still have $g(F_n/3)\to 1/3$, and so we can treat this case similarly to the case of $k=2$. We have $\alpha(3)=4$ and so we consider the Fibonacci numbers with index a multiple of $4$. We find that $g(F_{4n}/3)\to 1/3$. $F_4/3 = 1$ and so this case is trivial. For $n=2$, already we have $g(F_8/3) = g(7)\approx 0.326 \in [0,\tau^{-2}]$ and $|g(F_{4n}/3)-1/3|$ monotonically decreases, so we are always in the case that we can apply Theorem \ref{THM:adformula}. Also, we have $0.326 \leq g(F_{4n}/3) \leq 0.341$ by the above monotonicity. It follows that, for $n \geq 2$, we have
\[
A(F_{4n}/3) = \left\lceil \frac{\tau^{-1}}{g(F_{4n}/3)}\right\rceil = 2,
\]
because the inequality $0.326 \leq g(F_{4n}/3) \leq 0.341$ gives $1.812 \leq \tau^{-1}/g(F_{4n}/3)\leq 1.896$.

We can check this in {\walnut}:
\begin{itemize}
\item {\ttfamily reg fib4n msd\_fib "0*100(0000)*";\\\# This accepts $F_{4n}$ for all $n>0$.}
\item {\ttfamily def fib4nby3 "?msd\_fib Ek (\$fib4n(k) \& k=3*n)";\\\# This only accepts $F_{4n}/3$ for all $n>0$.}
\item {\ttfamily eval isfibcheck "?msd\_fib An (n>=0) => (\$fib4nby3(n) \& \$fiboyes(n))";\\\# Checks if $F_{4n}/3$ is $F_m$ for some $m$; returns FALSE.}
\item {\ttfamily eval fib4nby3ad2check "?msd\_fib Ad (d>0 \& \$fib4nby3(d) => \$fibAd2(d))";\\\# Checks if all $A(F_{4n}/3)=2$; returns TRUE.}
\end{itemize}

We may similarly use {\walnut} to get the following list of $A(F_n/k)$ values for the first few $k$, which would be laborious to prove using the methods we employed in the $k=2$ and $k=3$ case because $g(d)$ can (a priori) take more than one value\footnote{Nevertheless, sometimes, as in the case for $k=5$, even though $g(F_n/k)$ has more than one cluster point, $A(F_n/k)$ still only takes one value.}:

\begin{prop}
We have
\begin{itemize}
\item $A(F_{4n}/3)=2$ for all $n \geq 1$
\item $A(F_{6n}/4)=3$ for all $n \geq 1$
\item $A(F_{5n}/5)=4$ for all $n \geq 2$ \qquad (and $A(F_5/5)=2$)
\item $A(F_{12n}/6)=4$ for all $n \geq 1$
\item $A(F_{8n}/7)=5$ for all $n \geq 1$
\item $A(F_{6n}/8)\in\{2,5\}$ for all $n \geq 1$ \quad (in particular, $A(F_{12n-6}/8)=2$ and $A(F_{12n}/8)=5$)
\item $A(F_{12n}/9)=6$ for all $n \geq 1$
\item $A(F_{15n}/10)\in\{3,7\}$ for all $n \geq 1$ \quad (in particular, $A(F_{30n-15}/10)=3$ and $A(F_{30n}/10)=7$)
\end{itemize}
\end{prop}
So, the first $k$ for which $A(F_n/k)$ takes more than one value infinitely often is $k=8$.

The reason we were so interested in finding families of $d$ for which $A(d)$ is constant or at least takes only finitely many values was two-fold: first, if $A(d)$ is constant for the family, then one can use {\walnut} to prove properties about the entire family simultaneously (for instance finding automata that describe $i(d)$, the earliest index at which a longest arithmetic progression appears); secondly, it was hoped that if enough values of $A(d)$ could be shown to be small, then it would be possible to improve the asymptotic bounds on $A(d)$. Unfortunately, we were not able to realise the second point and so we leave this as a conjecture.
\begin{conj}\label{CONJ:fibmaxslope}
For all $d \geq 1$,
\[
\frac{A(d)-1}{d} < \frac{\sqrt{5}}{\tau}.
\]
\end{conj}
The appearance of `$-1$' in the numerator takes into account that while $A(F_n)/F_n \to \sqrt{5}\tau^{-1}$, this convergence is from above. So these ratios appear to overestimate the slope with error uniformly bounded above by $1$. For non-Fibonacci values of $d$, this correction term should not be necessary.

\subsection{The step distance ranges}
We now pivot to investigate the function $g(d)$ in more detail.
\begin{cor}\label{COR:ad4andeven}
If $\tau^{-2}< g(d)<0.5$, then $A(d)\geq4$ and is even.
\end{cor}
\begin{proof}
This follows from the proof of Theorem \ref{THM:adformula2}.
\end{proof}

In Figure \ref{FIG:plotadvgd}, $A(d)$ values are plotted against the possible range of $g(d)$ with $0.001$ increment in $g(d)$ values which results in a `discrete-hyperbolic' graph.
\begin{figure}
  \centering
  \begin{tikzpicture}
    \begin{axis}[
      width=0.99\textwidth,    
      height=0.6\textwidth,   
      xlabel={$g(d)$},
      ylabel={$A(d)$},
      legend style={at={(0.5,-0.1)}, anchor=north, legend columns=-1},
      mark size=0.3pt,          
      ]
      
      \addplot[
        only marks,            
        mark=*,             
        color=black,
        ] 
        coordinates {
          (0.01, 62.0)
(0.011, 57.0)
(0.012, 52.0)
(0.013, 48.0)
(0.014, 45.0)
(0.015, 42.0)
(0.016, 39.0)
(0.017, 37.0)
(0.018, 35.0)
(0.019, 33.0)
(0.02, 31.0)
(0.021, 30.0)
(0.022, 29.0)
(0.023, 27.0)
(0.024, 26.0)
(0.025, 25.0)
(0.026, 24.0)
(0.027, 23.0)
(0.028, 23.0)
(0.029, 22.0)
(0.03, 21.0)
(0.031, 20.0)
(0.032, 20.0)
(0.033, 19.0)
(0.034, 19.0)
(0.035, 18.0)
(0.036, 18.0)
(0.037, 17.0)
(0.038, 17.0)
(0.039, 16.0)
(0.04, 16.0)
(0.041, 16.0)
(0.042, 15.0)
(0.043, 15.0)
(0.044, 15.0)
(0.045, 14.0)
(0.046, 14.0)
(0.047, 14.0)
(0.048, 13.0)
(0.049, 13.0)
(0.05, 13.0)
(0.051, 13.0)
(0.052, 12.0)
(0.053, 12.0)
(0.054, 12.0)
(0.055, 12.0)
(0.056, 12.0)
(0.057, 11.0)
(0.058, 11.0)
(0.059, 11.0)
(0.06, 11.0)
(0.061, 11.0)
(0.062, 10.0)
(0.063, 10.0)
(0.064, 10.0)
(0.065, 10.0)
(0.066, 10.0)
(0.067, 10.0)
(0.068, 10.0)
(0.069, 9.0)
(0.07, 9.0)
(0.071, 9.0)
(0.072, 9.0)
(0.073, 9.0)
(0.074, 9.0)
(0.075, 9.0)
(0.076, 9.0)
(0.077, 9.0)
(0.078, 8.0)
(0.079, 8.0)
(0.08, 8.0)
(0.081, 8.0)
(0.082, 8.0)
(0.083, 8.0)
(0.084, 8.0)
(0.085, 8.0)
(0.086, 8.0)
(0.087, 8.0)
(0.088, 8.0)
(0.089, 7.0)
(0.09, 7.0)
(0.091, 7.0)
(0.092, 7.0)
(0.093, 7.0)
(0.094, 7.0)
(0.095, 7.0)
(0.096, 7.0)
(0.097, 7.0)
(0.098, 7.0)
(0.099, 7.0)
(0.1, 7.0)
(0.101, 7.0)
(0.102, 7.0)
(0.103, 7.0)
(0.104, 6.0)
(0.105, 6.0)
(0.106, 6.0)
(0.107, 6.0)
(0.108, 6.0)
(0.109, 6.0)
(0.11, 6.0)
(0.111, 6.0)
(0.112, 6.0)
(0.113, 6.0)
(0.114, 6.0)
(0.115, 6.0)
(0.116, 6.0)
(0.117, 6.0)
(0.118, 6.0)
(0.119, 6.0)
(0.12, 6.0)
(0.121, 6.0)
(0.122, 6.0)
(0.123, 6.0)
(0.124, 5.0)
(0.125, 5.0)
(0.126, 5.0)
(0.127, 5.0)
(0.128, 5.0)
(0.129, 5.0)
(0.13, 5.0)
(0.131, 5.0)
(0.132, 5.0)
(0.133, 5.0)
(0.134, 5.0)
(0.135, 5.0)
(0.136, 5.0)
(0.137, 5.0)
(0.138, 5.0)
(0.139, 5.0)
(0.14, 5.0)
(0.141, 5.0)
(0.142, 5.0)
(0.143, 5.0)
(0.144, 5.0)
(0.145, 5.0)
(0.146, 5.0)
(0.147, 5.0)
(0.148, 5.0)
(0.149, 5.0)
(0.15, 5.0)
(0.151, 5.0)
(0.152, 5.0)
(0.153, 5.0)
(0.154, 5.0)
(0.155, 4.0)
(0.156, 4.0)
(0.157, 4.0)
(0.158, 4.0)
(0.159, 4.0)
(0.16, 4.0)
(0.161, 4.0)
(0.162, 4.0)
(0.163, 4.0)
(0.164, 4.0)
(0.165, 4.0)
(0.166, 4.0)
(0.167, 4.0)
(0.168, 4.0)
(0.169, 4.0)
(0.17, 4.0)
(0.171, 4.0)
(0.172, 4.0)
(0.173, 4.0)
(0.174, 4.0)
(0.175, 4.0)
(0.176, 4.0)
(0.177, 4.0)
(0.178, 4.0)
(0.179, 4.0)
(0.18, 4.0)
(0.181, 4.0)
(0.182, 4.0)
(0.183, 4.0)
(0.184, 4.0)
(0.185, 4.0)
(0.186, 4.0)
(0.187, 4.0)
(0.188, 4.0)
(0.189, 4.0)
(0.19, 4.0)
(0.191, 4.0)
(0.192, 4.0)
(0.193, 4.0)
(0.194, 4.0)
(0.195, 4.0)
(0.196, 4.0)
(0.197, 4.0)
(0.198, 4.0)
(0.199, 4.0)
(0.2, 4.0)
(0.201, 4.0)
(0.202, 4.0)
(0.203, 4.0)
(0.204, 4.0)
(0.205, 4.0)
(0.206, 4.0)
(0.207, 3.0)
(0.208, 3.0)
(0.209, 3.0)
(0.21, 3.0)
(0.211, 3.0)
(0.212, 3.0)
(0.213, 3.0)
(0.214, 3.0)
(0.215, 3.0)
(0.216, 3.0)
(0.217, 3.0)
(0.218, 3.0)
(0.219, 3.0)
(0.22, 3.0)
(0.221, 3.0)
(0.222, 3.0)
(0.223, 3.0)
(0.224, 3.0)
(0.225, 3.0)
(0.226, 3.0)
(0.227, 3.0)
(0.228, 3.0)
(0.229, 3.0)
(0.23, 3.0)
(0.231, 3.0)
(0.232, 3.0)
(0.233, 3.0)
(0.234, 3.0)
(0.235, 3.0)
(0.236, 3.0)
(0.237, 3.0)
(0.238, 3.0)
(0.239, 3.0)
(0.24, 3.0)
(0.241, 3.0)
(0.242, 3.0)
(0.243, 3.0)
(0.244, 3.0)
(0.245, 3.0)
(0.246, 3.0)
(0.247, 3.0)
(0.248, 3.0)
(0.249, 3.0)
(0.25, 3.0)
(0.251, 3.0)
(0.252, 3.0)
(0.253, 3.0)
(0.254, 3.0)
(0.255, 3.0)
(0.256, 3.0)
(0.257, 3.0)
(0.258, 3.0)
(0.259, 3.0)
(0.26, 3.0)
(0.261, 3.0)
(0.262, 3.0)
(0.263, 3.0)
(0.264, 3.0)
(0.265, 3.0)
(0.266, 3.0)
(0.267, 3.0)
(0.268, 3.0)
(0.269, 3.0)
(0.27, 3.0)
(0.271, 3.0)
(0.272, 3.0)
(0.273, 3.0)
(0.274, 3.0)
(0.275, 3.0)
(0.276, 3.0)
(0.277, 3.0)
(0.278, 3.0)
(0.279, 3.0)
(0.28, 3.0)
(0.281, 3.0)
(0.282, 3.0)
(0.283, 3.0)
(0.284, 3.0)
(0.285, 3.0)
(0.286, 3.0)
(0.287, 3.0)
(0.288, 3.0)
(0.289, 3.0)
(0.29, 3.0)
(0.291, 3.0)
(0.292, 3.0)
(0.293, 3.0)
(0.294, 3.0)
(0.295, 3.0)
(0.296, 3.0)
(0.297, 3.0)
(0.298, 3.0)
(0.299, 3.0)
(0.3, 3.0)
(0.301, 3.0)
(0.302, 3.0)
(0.303, 3.0)
(0.304, 3.0)
(0.305, 3.0)
(0.306, 3.0)
(0.307, 3.0)
(0.308, 3.0)
(0.309, 3.0)
(0.31, 2.0)
(0.311, 2.0)
(0.312, 2.0)
(0.313, 2.0)
(0.314, 2.0)
(0.315, 2.0)
(0.316, 2.0)
(0.317, 2.0)
(0.318, 2.0)
(0.319, 2.0)
(0.32, 2.0)
(0.321, 2.0)
(0.322, 2.0)
(0.323, 2.0)
(0.324, 2.0)
(0.325, 2.0)
(0.326, 2.0)
(0.327, 2.0)
(0.328, 2.0)
(0.329, 2.0)
(0.33, 2.0)
(0.331, 2.0)
(0.332, 2.0)
(0.333, 2.0)
(0.334, 2.0)
(0.335, 2.0)
(0.336, 2.0)
(0.337, 2.0)
(0.338, 2.0)
(0.339, 2.0)
(0.34, 2.0)
(0.341, 2.0)
(0.342, 2.0)
(0.343, 2.0)
(0.344, 2.0)
(0.345, 2.0)
(0.346, 2.0)
(0.347, 2.0)
(0.348, 2.0)
(0.349, 2.0)
(0.35, 2.0)
(0.351, 2.0)
(0.352, 2.0)
(0.353, 2.0)
(0.354, 2.0)
(0.355, 2.0)
(0.356, 2.0)
(0.357, 2.0)
(0.358, 2.0)
(0.359, 2.0)
(0.36, 2.0)
(0.361, 2.0)
(0.362, 2.0)
(0.363, 2.0)
(0.364, 2.0)
(0.365, 2.0)
(0.366, 2.0)
(0.367, 2.0)
(0.368, 2.0)
(0.369, 2.0)
(0.37, 2.0)
(0.371, 2.0)
(0.372, 2.0)
(0.373, 2.0)
(0.374, 2.0)
(0.375, 2.0)
(0.376, 2.0)
(0.377, 2.0)
(0.378, 2.0)
(0.379, 2.0)
(0.38, 2.0)
(0.381, 2.0)
(0.38196601125, 2.0)
(0.382, 4.0)
(0.383,4)
(0.384,4)
(0.385,4)
(0.386,4)
(0.387,4)
(0.388,4)
(0.389,4)
(0.390,4)
(0.391,4)
(0.392,4)
(0.393,4)
(0.394,4)
(0.395,4)
(0.396,4)
(0.397,4)
(0.398,4)
(0.399,4)
(0.400,4)
(0.401,4)
(0.402,4)
(0.403,4)
(0.404,4)
(0.405,4)
(0.406,4)
(0.407,4)
(0.408,4)
(0.409,4)
(0.410,4)
(0.411,4)
(0.412,4)
(0.413,4)
(0.414,4)
(0.415,4)
(0.416,4)
(0.417,4)
(0.418,4)
(0.419,4)
(0.420,4)
(0.421,4)
(0.422,4)
(0.423,4)
(0.424,4)
(0.425,4)
(0.426,4)
(0.427,4)
(0.428,4)
(0.429,4)
(0.430,4)
(0.431,4)
(0.432,4)
(0.433,4)
(0.434,4)
(0.435,4)
(0.436,4)
(0.437,4)
(0.438,4)
(0.439,4)
(0.440,4)
(0.441,4)
(0.442,4)
(0.443,4)
(0.444,4)
(0.445,4)
(0.446,4)
(0.447,4)
(0.448,4)
(0.449,4)
(0.450,4)
(0.451,4)
(0.452,4)
(0.453,4)
(0.454,4)
(0.455,4)
(0.456,4)
(0.457,4)
(0.458,4)
(0.459,4)
(0.460,4)
(0.461,6)
(0.462,6)
(0.463,6)
(0.464,6)
(0.465,6)
(0.466,6)
(0.467,6)
(0.468,6)
(0.469,6)
(0.470,6)
(0.471,6)
(0.472,6)
(0.473,6)
(0.474,6)
(0.475,6)
(0.476,8)
(0.477,8)
(0.478,8)
(0.479,8)
(0.480,8)
(0.481,8)
(0.482,8)
(0.483, 10.0)
(0.484, 10.0)
(0.485, 10.0)
(0.486, 10.0)
(0.487, 12.0)
(0.488, 12.0)
(0.489, 14.0)
(0.490, 14.0)
(0.491, 16.0)
(0.492, 18.0)
(0.493, 20.0)
(0.494, 24.0)
(0.495, 28.0)
(0.496, 36.0)
(0.497, 48.0)
        };
    \end{axis}
  \end{tikzpicture}
  \caption{Plot of $A(d)$ for the entire range of $g(d)$.}
  \label{FIG:plotadvgd}
\end{figure}
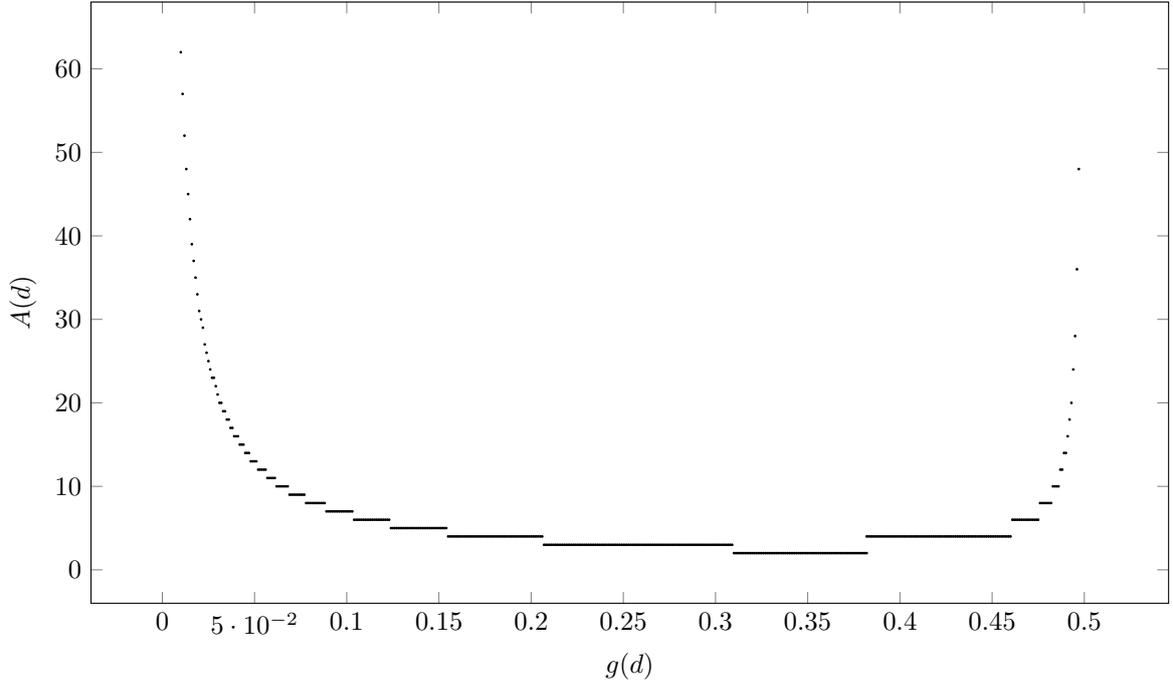
We can determine the range of $g(d)$ for which $A(d)$ is some integer $k\geq2$ using the formulae for $A(d)$ from Theorems \ref{THM:adformula} and \ref{THM:adformula2}. That is, we can calculate the level sets of $A(d)$ as a function of $g(d)$.

For  $A(d)=2$, using Corollary \ref{COR:ad4andeven}, it is clear that $g(d)=\tau^{-2}$ is the maximal possible value of $g(d)$. For the minimum, we use the formula $A(d)=\left\lceil \frac{\tau^{-1}}{g(d)}\right\rceil$ and so the minimum possible value of $g(d)$ is $g(d)=\tau^{-1}/2$. Recall from Remark \ref{REM:I0gdnodiv}, that $g(d)$ never completely divides $I_0$, so the left end point is excluded from the set of possible values of $g(d)$. Hence we have $g(d)\in(\tau^{-1}/2,\tau^{-2}]$.

Similarly, the level set for $A(d)=3$ can be shown to be $g(d)\in(\tau^{-1}/3,\tau^{-1}/2)$. Even values of $A(d)$ are slightly more involved because there are two disjoint intervals that comprise the level set; one in which $g(d)<\tau^{-2}$ and one in which $g(d)>\tau^{-2}$. For example, when $A(d)=4$, along with $g(d)\in(\tau^{-1}/4,\tau^{-1}/3)$, we also have the interval $g(d)\in\left(\tau^{-2},\frac{2-\tau^{-1}}{3}\right)$ using the formula for $A(d)$ from Theorem \ref{THM:adformula2}.

The general result is the following.
\begin{prop}\label{PROP:gdranges}
Let $k\geq3$. If $k$ is odd, then $A(d)=k$ if and only if \[g(d)\in\left(\frac{\tau^{-1}}{k},\frac{\tau^{-1}}{k-1}\right);\]
and if $k$ is even, then $A(d)=k$ if and only if \[g(d)\in\left(\frac{\tau^{-1}}{k},\frac{\tau^{-1}}{k-1}\right)\cup\left(\frac{k-2\tau}{2k-6},\frac{k-2\tau^{-1}}{2k-2}\right).\]
\end{prop}
\begin{proof}
    For this, we consider the range of $g(d)$ and use the corresponding $A(d)$ formulae given in Theorems \ref{THM:adformula} and \ref{THM:adformula2} respectively. In particular, we first substitute $A(d)=k$ into each formula and remove the ceiling to find the lower limit of $g(d)$. Then for the upper limit, we repeat the process but instead substitute $A(d)=k+1$. Since $A(d)$ is always even for $g(d)>\tau^{-2}$ (Corollary \ref{COR:ad4andeven}), only one formula for $A(d)$ applies in the case that $A(d)=k$ is odd, giving rise to a single interval coming from Theorem \ref{THM:adformula}.

    In the case that $A(d)$ is even, there are corresponding values of $g(d)$ in both of the intervals $[0,\tau^{-2})$ and $[\tau^{-2},0.5]$. Whence, we obtain a union of two disjoint $g(d)$ ranges when $A(d)=k$ is even, corresponding to the two different formulae for $A(d)$ coming from both Theorem \ref{THM:adformula} and Theorem \ref{THM:adformula2}.
\end{proof}
\begin{rem}Notice that since the ranges of $g(d)$ except for the first case of $A(d)=2$ are always open. This means that for $k\geq2$, the isolated points $\tau^{-1}/k$ and $\frac{k-2\tau}{2k-6}$ on the range of $g(d)$ are avoided.
\end{rem}
An immediate consequence of this result is that $A(d)$ takes all possible values and indeed infinitely often because the sequence $g(d)$ is dense in $[0,0.5]$.
\begin{prop}
For all $k \geq 2$, there exist infinitely many $d \geq 1$ such that $A(d)=k$.\qed
\end{prop}

We can further show the density of values of $d$ for which $A(d)=k$ for a fixed $k\geq 1$.
Let
\[
\mathrm{dens}_{\,\bff}(k)=\lim_{n\to\infty}\frac{\#\left\{d\in\{1,\dots ,n\} \mid A(d)=k\right\}}{n}.
\]
In order to calculate the density, we first need to show that not only is the orbit of $g(d)$ dense in $[0,0.5]$, but it is also uniformly distributed.
\begin{prop}\label{LEM:gdisud}
The values of $g(d)$ are uniformly distributed in the interval $[0,0.5]$.
\end{prop}
\begin{proof}
Recall the definition of step distance, $g(d)=\min\{\{ d\tau\}, 1-\{ d\tau\}\}$. So we may interpret $g$ as a function $g \colon \N \to [0,0.5]$.
We know that the function $r_{\tau} \colon \N \to [0,1]$ given by $r_\tau \colon d \mapsto \{ d\tau\}$ is uniformly distributed in $[0,1]$ by Weyl's theorem, as $\tau$ is irrational. Now let $h\colon [0,1] \to [0,0.5]$ denote the function
\[h(x)=\begin{cases}
x\,&\text{ if }\:0<x<0.5,\\
1-x\,&\text{ if }\:0.5<x<1,
\end{cases}\]
which `folds' the unit interval in half. We have $g = h \circ r_\tau$. As $r_\tau$ is uniformly distributed and $h(x)$ preserves Lebesgue measure, the uniform distribution of their composition follows.

\end{proof}
So now, the density of $A(d)$ can be calculated by dividing the length of its level set (coming from Proposition \ref{PROP:gdranges}) by 0.5, the total length of the codomain of $g$.
For example, the density of the set of values of $d$ for which $A(d)=3$ is
\[\mathrm{dens}_{\,\bff}(3)=2\left(\frac{\tau^{-1}}{2}-\frac{\tau^{-1}}{3}\right)\approx 0.2060.\]
The general result is then as follows.
\begin{theorem}
Let $k\geq3$. If $k$ is odd, then 
\[
\mathrm{dens}_\mathbf{\,\bff} = 2\left(\frac{\tau^{-1}}{k-1} - \frac{\tau^{-1}}{k}\right).
\]
If $k$ is even, then
\[
\mathrm{dens}_\mathbf{\,\bff} = 2\left(\frac{\tau^{-1}}{k-1} - \frac{\tau^{-1}}{k}
+ \frac{k-2\tau^{-1}}{2k-2} - \frac{k-2\tau}{2k-6}\right).
\]
\hfill\qed
\end{theorem}

Using the formulae above, we determine that the most frequent value of $A(d)$ is $4$ with $\mathrm{dens}_{\,\bff}(4)=2\left(\frac{\tau^{-1}}{3}-\frac{\tau^{-1}}{4}+\frac{2-\tau^{-1}}{3}-\tau^{-2}\right)\approx 0.2604$.

\subsection{Alternative proofs of some of the \walnut{}-proven statements}\label{SEC:altproofs}
To conclude, we use the techniques developed in this section to provide alternative proofs for some of the results from Section \ref{SEC:results-walnut} that were proven using \walnut{}.

First, we provide an alternative proof for Proposition \ref{propAd2i02}: for all $d$, if $A(d)=2$, then either $i(d)=0$ or $i(d)=2$. Moreover, we are able to classify exactly when $i(d)=0$ and when $i(d)=2$.
\begin{prop}\label{PROP:ad2i02details}
If $A(d)=2$, and $\{ d\tau\}<0.5$ then $i(d)=0$. If $A(d)=2$ and $\{ d\tau\}>0.5$ then $i(d)=2$.
\end{prop}
\begin{proof}
For all $d$, if $A(d)=2$, then by Proposition \ref{PROP:gdranges}, the possible step distance range is $g(d)\in(\tau^{-1}/2,\tau^{-2}]$.

Now, there are two ways we can have $g(d)$ in the above range. From Figure \ref{FIG:stepdisEQ}, recall if $\{ d\tau\} < 0.5$ and if the points $x_n = \{ nd\tau\}$ and $x_{n+1} = \{ (n+1)d\tau\}$ are both in $I_0$, then $x_{n+1}>x_n$, so we are stepping to the right; otherwise we are stepping to the left.
We should determine the interval of starting positions within $I_0$ of orbits that correspond to longest arithmetic progressions of length exactly $2$. We consider each case in turn.

\textbf{Case 1} ($\{ d\tau\} < 0.5$): In this case, we are stepping to the right. Therefore, such a starting position must satisfy (see Figure \ref{FIG:whenad2i0}) the following conditions: (1) we start in $I_0$, (2) by stepping to the left we enter $I_1$, (3) we are able to step exactly once to the right remaining in $I_0$, (4) a second step places us outside of $I_0$. Let $\hat{x}\in I_0$ be such a starting point. Hence, $\hat{x}$ must satisfy the following four inequalities
\begin{enumerate}
\item[(i)] $\tau^{-2}<\hat{x}<1$
\item[(ii)] $\hat{x}-g(d)<\tau^{-2}$
\item[(iii)] $\hat{x}+g(d)<1$
\item[(iv)] $\hat{x}+2g(d)>1$
\end{enumerate}

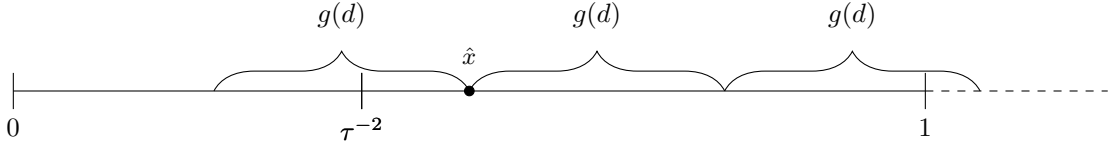
\begin{figure}[ht]
\centering
\begin{tikzpicture}[scale=12]
    \draw (0,0)-- (0.382,0);
    \draw (0.382,0)-- (1,0);
    \draw (1,0)-- (1.2,0) [dashed];
    \foreach \x in {0,1} {
        \draw (\x,0.02) -- (\x,-0.02) node[below] {\x};
        \draw (0.382,0.02) -- (0.382,-0.02) node[below] {$\tau^{-2}$};
    }
    \filldraw[black] (0.5,0) circle (0.15pt);
    \node[above] at (0.5,0.02) {\small $\hat{x}$};
     \draw[decorate,decoration={brace,amplitude=15pt}] (0.5,0) -- (0.78,0) node[midway,above,yshift=20pt] {$g(d)$};
     \draw[decorate,decoration={brace,amplitude=15pt}] (0.78,0) -- (1.06,0) node[midway,above,yshift=20pt] {$g(d)$};
     \draw[decorate,decoration={brace,amplitude=15pt}] (0.5-0.28,0) -- (0.5,0) node[midway,above,yshift=20pt] {$g(d)$};
\end{tikzpicture}
\caption{A potential starting position $\hat{x}$ corresponding to the case $A(d) = 2$ and $\{ d\tau \}<0.5$.}
\label{FIG:whenad2i0}
\end{figure}

Figure \ref{FIG:whenad2i0} illustrates which subinterval of $I_0$ we are considering to find a suitable $\hat{x}$. Solving the system of inequalities, together with the fact that $\tau^{-1}/2 \leq g(d) \leq \tau^{-2}$, we get that $j\in(\tau^{-2},1-\frac{\tau^{-1}}{2})$. The least value of $n$ for which $\{ n\tau \} \in (\tau^{-2},1-\frac{\tau^{-1}}{2})$ is $n=1$. Hence, $i(d)=0$.

\textbf{Case 2} ($\{ d\tau\} > 0.5$): In this case, we are stepping to the left, otherwise it works similarly. We find that a starting position must satisfy the following conditions: (1) we start in $I_0$, (2) by stepping to the right we enter $I_1$, (3) we are able to step exactly once to the left remaining in $I_0$, (4) a second step places us outside of $I_0$. Let $\hat{x}\in I_0$ be such a starting point. Hence, $\hat{x}$ must satisfy the following four inequalities
\begin{enumerate}
\item[(i)] $\tau^{-2}<\hat{x}<1$
\item[(ii)] $1-\hat{x}<g(d)$
\item[(iii)] $\hat{x}-g(d)>\tau^{-2}$
\item[(iv)] $\hat{x}-2g(d)<\tau^{-2}$
\end{enumerate}
Solving the system of inequalities, together with the fact that $\tau^{-1}/2 \leq g(d) \leq \tau^{-2}$, we get that $\hat{x}\in(1-\frac{\tau^{-1}}{2},1)$. The least value of $n$ for which $\{ n\tau \} \in (1-\frac{\tau^{-1}}{2},1)$ is $n=3$. Hence, $i(d)=2$.
\end{proof}
\begin{rem}
Notice that, in Figure \ref{fibad2fig}, the accepting states {\ttfamily 1} and {\ttfamily 7} correspond to the values of $d$ for which $A(d)=2$. \walnut{} can be used to show that the accepting state {\ttfamily 1} corresponds to the values of $d$ where $i(d)=2$, and  the accepting state {\ttfamily 7} corresponds to the values of $d$ where $i(d)=0$. Also, from Proposition \ref{PROP:ad2i02details}, the accepting states {\ttfamily 1} and {\ttfamily 7} equivalently correspond to the values of $d$ for which $\{ d\tau\}>0.5$ and $\{ d\tau\}<0.5$ respectively.
\end{rem}
\begin{prop}[Partial converse of Proposition \ref{propAd2i02}]\label{PROP:id0Ad2}
For all $d$, if $i(d)=0$ then $A(d)=2$.
\end{prop}
\begin{proof}
Suppose, $A(d)\geq3$. Then Proposition \ref{PROP:gdranges} implies $g(d)\in(0,\tau^{-1}/2]\cup(\tau^{-2},0.5)$. If $i(d)=0$, that is, if one begins at the point $\{ \tau \}$, then we have two cases: firstly, for $g(d)\in(0,\tau^{-1}/2]$, there is always room for more points to fit inside $I_0$ by shifting the first point to either side, which means $i(d)$ cannot be $0$; secondly, for $g(d)\in(\tau^{-2},0.5)$, when one begins at $\{ \tau \}$, the next point will always appear in $I_1$, which implies that $A(d)=1$, which is a contradiction.
\end{proof}

\begin{prop}[Part of Proposition \ref{PROP:dfn+2fib}]\label{PROP:ad3id3}
For $d=F_n+2$ for some $n\geq8$, $A(d)=i(d)=3$.
\end{prop}
\begin{proof}
We have from Corollary \ref{COR:ceiltauad} that $\{ 2\tau \} = g(F_3)=\tau^{-3}$ and $g(F_8)=\tau^{-8}$. So for $d=F_n+2$ for $n\geq8$, we have $g(F_n+2)=\{ (F_n+2)\tau\} \in [\tau^{-3}-\tau^{-8},\tau^{-3}+\tau^{-8}]$. This is because $\{ F_{2n}\tau \}\to1$ from below and $\{ F_{2n+1}\tau \}\to0$ from above, both monotonically. 

Since $g(F_n+2)\leq\tau^{-2}$, we use the formula for $A(d)$ from Theorem \ref{THM:adformula}. That is,
\[A(F_n+2) =\left\lceil\frac{\tau^{-1}}{g(F_n+2)}\right\rceil = 3.\]
Since $\{ (F_n+2)\tau\}<0.5$, we always step to the right when counting the MAP (again, see Figure \ref{FIG:stepdisEQ}). Hence we find a candidate starting point $\hat{x}\in I_0$ which must satisfy the following inequalities:
\begin{enumerate}
\item[(i)] $\tau^{-2}<\hat{x}<1$
\item[(ii)] $\hat{x}-g(F_n+2)<\tau^{-2}$
\item[(iii)] $\hat{x}+2g(F_n+2)<1$
\item[(iv)] $\hat{x}+3g(F_n+2)>1$
\end{enumerate}
Solving the system of inequalities, we get that $\hat{x}\in(\tau^{-2},\tau^{-1}+\tau^{-8})$. The least value of $n$ for which $\{ n\tau \} \in (\tau^{-2},\tau^{-1}+\tau^{-8})$ is $n=4$. Hence, $i(d)=3$.
\end{proof}
In principle, one can prove similar statements for other families of $d$ using analogous methods, although the arguments will become ever increasingly complicated. One can therefore further partition the $g(d)$ intervals identified in Proposition \ref{PROP:gdranges} into subintervals on which $i(d)$ is constant. For example, it can be shown that $A(d)=3$ and $i(d)=3$ if and only if $g(d)\in\left(\frac{\tau^{-1}}{3},\frac{1-\{ 4\tau\}}{2}\right)$. The complete general result in this direction requires a careful analysis of the first returns of the orbits $(\{nd\tau\})_{n \geq 1}$ to particular subintervals, the full formulation of which is somewhat more technical and is therefore reserved for the first author's future output.

\section*{Acknowledgements}
The authors would like to thank Ibai Aedo for sharing unpublished notes, Jamie Walton for discussions surrounding the three-gaps theorem, and Jarkko Peltom\"aki for valuable discussions relating to rotation sequences. GJ is funded by the Open University, UK.

\section*{Declarations}
GJ received research support from the Open University, UK.
DR has no financial interests to declare.

\end{document}